\documentclass[english,11pt,a4paper,leqno]{amsart}
\usepackage{amsmath,amstext, amsthm, amssymb}
\usepackage{latexsym, amsthm, amsfonts, amsmath,amssymb}
\usepackage[latin1]{inputenc} 
\usepackage[T1]{fontenc} 
\usepackage{enumerate}
\usepackage[english]{babel} 

\usepackage
           {geometry}
\usepackage{ifxetex,ifluatex}
\newif\ifxetexorluatex
\ifxetex
  \xetexorluatextrue
\else
  \ifluatex
    \xetexorluatextrue
  \else
    \xetexorluatexfalse
  \fi
\fi

\ifxetexorluatex
  \usepackage{fontspec}           
  \usepackage{unicode-math}       
  \usepackage{polyglossia}        
  \setmainlanguage{english}
\else
  \usepackage[latin1]{inputenc}     
  \usepackage[english]{babel}     

  \usepackage{amsfonts}
  \usepackage{amssymb}
  
  \usepackage{lmodern}
  \usepackage[T1]{fontenc}
\fi
\usepackage{booktabs}   
\usepackage[inline]{enumitem} 
\usepackage{xspace}     
\usepackage[svgnames]{xcolor}     
\usepackage{amsmath}    
\usepackage{mathtools}  
\usepackage{stmaryrd} 

\usepackage{graphics}
\usepackage{etoolbox}
\usepackage{blkarray}
\usepackage{tikz-cd}

\definecolor{linkblue}{RGB}{1,1,190}
\definecolor{citegreen}{RGB}{1,190,1}
\usepackage[linkcolor=linkblue
           ,citecolor=citegreen
           ,ocgcolorlinks        
           ,bookmarksopen=True,  
           ]{hyperref}
\usepackage[ocgcolorlinks]{ocgx2} 

\makeatletter
  \AtBeginDocument{
    \hypersetup{
      pdftitle  = {\@title},
      pdfauthor = {\authors}     
    }
  }

\makeatother

\usepackage{amsthm}

\usepackage[capitalize    
           ]{cleveref}

\theoremstyle{definition}
\newtheorem {definition}{Definition}[section]

\theoremstyle{plain}
\newtheorem {theorem}[definition]{Theorem}
\crefname   {theorem}{Theorem}{Theorems}
\newtheorem*{theorem*}{Theorem}
\newtheorem {lemma}[definition]{Lemma}
\newtheorem {proposition}[definition]{Proposition}

\newtheorem {corollary}[definition]{Corollary}

\theoremstyle{remark}
\newtheorem {remark}[definition]{Remark}
\newtheorem {example}[definition]{Example}
\newtheorem*{example*}{Example}
\crefname   {example}{Example}{Examples}

\AtBeginEnvironment{smallremark}{\footnotesize}

\makeatletter

\newcommand{\sc@lettershortcut}[3]{%
  \expandafter\providecommand\csname #2#3\endcsname{#1{#3}}%
}

\newcommand{\sc@shortcuts}[3]{%
  \count@=0
  \loop
  \advance\count@ 1
  \edef\tmp@{%
    \noexpand\sc@lettershortcut\unexpanded{{#1}}{#2}{#3\count@}
  }
  \tmp@
  \ifnum\count@<26
  \repeat
}

\newcommand{\defshortcuts}[2]{\sc@shortcuts{#1}{#2}{\@alph}}

\newcommand{\defShortcuts}[2]{\sc@shortcuts{#1}{#2}{\@Alph}}

\makeatother
\defShortcuts{\mathbb}{b}
\defShortcuts{\mathcal}{c}
\defShortcuts{\mathfrak}{f}
\defShortcuts{\mathsf}{s}
\defshortcuts{\mathfrak}{f}
\defshortcuts{\mathsf}{s}

\def\rfop{*}
\makeatletter
\newcommand\rigidfactorization[2][]{%
  \def\rf@delim{\rfop}
  \newif\ifrf@notfirst
  #1
  \@for\next:=#2\do{%
    \ifrf@notfirst
      \rf@delim
    \fi
    \rf@notfirsttrue
    \next
  }%
}
\makeatother
\newcommand\rf\rigidfactorization
\ifxetexorluatex
  
\else
  
\fi

\DeclarePairedDelimiter{\length}{\lvert}{\rvert}
\DeclarePairedDelimiter{\abs}{\lvert}{\rvert}
\DeclarePairedDelimiter{\card}{\lvert}{\rvert}
\DeclarePairedDelimiter{\norm}{\lVert}{\rVert}

\DeclareMathOperator{\GL}{GL}

\DeclareMathOperator{\Norm}{N}

\setlist[itemize,1]{leftmargin=0.75cm}

\setlist[enumerate,1]{label=\textup{(\arabic*)}, ref=\textup{(}\arabic*\textup{)}, leftmargin=0.75cm}
\setlist[enumerate,2]{label=\textup{(}\roman*\textup{)}, ref=\textup{(}\roman*\textup{)}}

\newlist{equivenumerate}{enumerate}{1}
\setlist[equivenumerate,1]{%
  label=\textup{(\alph*)},
  ref=\textup{(}\alph*\textup{)},
  leftmargin=0.75cm
}

\newlist{equivenumerate*}{enumerate*}{1}
\setlist*[equivenumerate*,1]{%
  label=\textup{(\alph*)},
  ref=\textup{(}\alph*\textup{)},
  leftmargin=0.75cm
}

\newlist{propenumerate}{enumerate}{1}
\setlist[propenumerate,1]{%
  label=\textup{(\roman*)},
  ref=\textup{(}\roman*\textup{)},
  leftmargin=0.75cm
}

{\end{description}}

\makeatletter
\def\sr@stripleadingcol::#1{#1}
\def\sr@dosubref#1#2:#3 #4{\if\relax#3\relax%
  \def\first{\sr@stripleadingcol #4}%
  #1{\first}\ref{\first:#2}%
\else%
  \sr@dosubref#1#3 {#4:#2}%
\fi}%

\newcommand{\subref}[1]{\sr@dosubref\cref#1: :\relax}
\newcommand{\Subref}[1]{\sr@dosubref\Cref#1: :\relax}
\makeatother
\AtBeginDocument{\renewcommand{\setminus}{\smallsetminus}}

\numberwithin{equation}{section}
\crefname{equation}{Equation}{Equations}
\Crefname{equation}{Equation}{Equations}

\title{A height gap theorem for coefficients of Mahler functions}

\author{Boris Adamczewski}
\address{Univ Lyon, Universit\'e Claude Bernard Lyon 1, CNRS UMR 5208, Institut Camille Jordan, 43 blvd. du 11 novembre 1918, F-69622 Villeurbanne cedex, France}
\email{boris.adamczewski@math.cnrs.fr}
\author{Jason Bell}
\address{Department of Pure Mathematics, University of Waterloo, Waterloo, ON, Canada N2L 3G1}
\email{jpbell@uwaterloo.ca}

\author{Daniel Smertnig}
\address{University of Graz, Institute for Mathematics and Scientific Computing, NAWI Graz, Heinrichstrasse 36, 8010 Graz, Austria}
\email{daniel.smertnig@uni-graz.at}

\thanks{This project has received funding from the European Research Council (ERC) under the European Union's Horizon 2020 research and innovation program under the Grant Agreement No 648132. Smertnig was supported by the Austrian Science Fund (FWF) project J4079-N32.
Part of the research was conducted while Smertnig was visiting University of Waterloo; he would like to extend his thanks for their hospitality.}
\subjclass[2010]{}
\date{}

\DeclareMathOperator{\trdeg}{tr.deg}
\DeclareMathOperator{\ev}{ev}
\newcommand{\laurent}[1]{{#1}(\!(\!\,z\!\,)\!)}
\newcommand{\power}[1]{{#1}\llbracket z \rrbracket}

\newcommand{\Qbar}{\overline \bQ}
\newcommand{\digits}{\Sigma_k} 
\newcommand{\bba}{\mathbf a}

\renewcommand{\card}[1]{\#{#1}}
\newcommand{\ru}{\cU}  
\newcommand{\rucommon}{\cU_{k}} 

\begin{document}

\begin{abstract}
We study the asymptotic growth of coefficients of Mahler power series with algebraic coefficients, as 
measured by their logarithmic Weil height. 
We show that there are five different growth behaviors, all of which being reached. 
Thus, there are \emph{gaps} in the possible growths. 
In proving this height gap theorem, we obtain that a $k$-Mahler function is $
k$-regular if and only if  its coefficients have height in $O(\log n)$. 
Moreover, we deduce that, over an arbitrary ground field of characteristic zero, a 
$k$-Mahler function  is 
$k$-automatic if and only if its coefficients belong to a finite set.  
As a by-product of our results, we also recover a conjecture of Becker which was recently settled by 
Bell, Chyzak, Coons, and Dumas. 
\end{abstract}

\maketitle

\section{Introduction}

The study of power series solutions to linear differential equations with coefficients in $\Qbar[z]$ 
provides a deep interplay between various fields of mathematics and physics, including combinatorics 
and number theory. 
For instance, the study of generating series in enumerative combinatorics benefits from 
the useful  dictionary between asymptotics of coefficients of $D$-finite power series 
and the type of singularities of the corresponding differential equation (see \cite{FS-09}). 
More surprisingly, prescribing some kind of arithmetic behavior for coefficients gives rise to powerful number 
theoretical consequences, as first perceived by Siegel \cite{Siegel-29} when introducing $E$- and $G$-functions, 
and pursued more recently by Andr\'e \cite{andre-001,andre-002} in his study of arithmetic Gevrey series.

This paper deals with the arithmetic behavior of coefficients of \emph{Mahler functions}, 
or \emph{$M$-functions},  which are power series of a very different kind. 
Unless it is rational, an $M$-function never satisfies 
a linear or even an algebraic differential equation \cite{ADH}. 
Instead, $M$-functions are solutions to linear difference equations with coefficients in $\Qbar[z]$ 
associated with the Mahler operator $z\mapsto z^k$, where $k\geq 2$ is a natural number. 
Precisely, a power series $f(z) \in \power{\Qbar}$ is a \emph{$k$-Mahler function}, 
or for short \emph{$k$-Mahler},  
if it satisfies an equation of the form
\begin{equation}\label{eq: mahler1} 
 p_0(z) f(z) + \cdots + p_d(z) f(z^{k^d}) = 0 
\end{equation}
with $p_{0}(z)$, $\ldots\,$,~$p_d(z) \in \Qbar[z]$ and $p_0(z)p_d(z) \ne 0$. A power series is an $M$-function if it is a 
$k$-Mahler function for some $k$.  
The study of $M$-functions and their values was initiated at the end of the 1920's by Mahler 
\cite{mahler-29,mahler-302,mahler-301}, who 
developed a new direction in transcendence theory, nowadays known as Mahler's method.    
In fact, Mahler only considered order one equations, but possibly inhomogeneous and also non-linear ones. 
The interest for $M$-functions of arbitrary order really took on a new significance at the 
beginning of the 1980s 
after Mend\`es France popularized among number theorists a result of Cobham \cite{cobham-68} stating  
that \emph{automatic power series} 
 are $M$-functions.  After recent results \cite{philippon15,adamczewski-faverjon17}, the transcendence 
 theory of $M$-functions mirrors exactly 
the one of $E$-functions. 
Beyond Mahler's method and automata theory, it is worth mentioning that $M$-functions 
naturally occur as generating functions in various other topics such as combinatorics of partitions, 
numeration, and analysis of algorithms.  In particular, the \emph{regular power series} introduced 
by Allouche and Shallit  \cite{allouche-shallit92} form a distinguished class of  $M$-functions. There is also 
 a mysterious interplay between $G$-functions and $M$-functions that deserves more attention.  
Indeed, for some $G$-functions $\sum_{n=0}^\infty a_n z^n\in\power{\mathbb Q}$, the power series 
$\sum_{n=0}^\infty \sv_p(a_n)z^n$, where $\sv_p(a_n)$ 
is the $p$-adic valuation of $a_n$, turns out to be  $p$-Mahler.  
This is likely related to the fact that Picard-Fuchs differential equations 
have a strong Frobenius structure for almost all primes.  
In recent years, there is renewed interest in $M$-functions, as evidenced by the flourishing literature on this topic.   
The latter includes discussions on various perspectives such as  transcendence and algebraic independence, 
combinatorics and theoretical computer science, the study of Mahler's equations and associated Galois theories, 
and computational aspects.  A number of references can be found in the survey \cite{adamczewski-selecta}.
See also \cite{ADH,adamczewski-dreyfus-hardouin-wibmer20,adamczewski-faverjon20,bell-chyzak-coons-dumas18} for more recent ones.

\subsection{The Height Gap Theorem} 
Let us first recall that the coefficients of a $k$-Mahler function $\sum_{n=0}^\infty a_n z^n\in\power{\Qbar}$ 
satisfy some recurrence relation of the form 
\[
    a_n = \sum_{j=1}^s-\alpha_ja_{n-j} + \sum_{i=1}^d \sum_{j=0}^s \beta_{i,j} a_{\frac{n-j}{k^i}} \,,
  \]
where $\alpha_j$ and $\beta_{i,j}$ are algebraic numbers, and $n$ is large enough 
(see Equation \eqref{eq: recurrence}, in which one can achieve $m=0$ using \cref{l:kill-0}). 
It follows 
that the field extension of $\mathbb Q$ generated by all coefficients $a_n$ 
is a number field. 
In the sequel, we will measure the coefficients of an $M$-function   
by their logarithmic Weil height. 

\subsubsection{The logarithmic Weil height} 

For a number field, we normalize the non-trivial absolute values as in \cite{bombieri-gubler06}.
Thus, for $\bQ$ and $p$ a prime we let $\abs{p}_p=1/p$; for the archimedean place of $\bQ$ 
we use the usual absolute value.
For $K$ a number field and a place $\sw$ of $K$ extending a place $\sv$ of $\bQ$, let
\[
  \abs{\alpha}_\sw \coloneqq \abs{\Norm_{K_\sw/\bQ_\sv}(\alpha)}_\sv^{1/[K:\bQ]}.
\]
Then the set of places $M_K$ on $K$ satisfies the product formula.
For $\alpha \in \Qbar$ the \emph{logarithmic absolute Weil height} is defined by 
\[
  h(\alpha) = \log \prod_{\sv \in M_K} \max\{ 1, \abs{\alpha}_\sv \},
\]
where $K$ is any number field containing $\alpha$. The value $h(\alpha)$ in this 
definition does not depend on the choice of such a number field $K$.
For $a/b \in \bQ\setminus\{0\}$ with $a$,~$b \in \bZ$, $b \ne 0$,  and $\gcd(a,b)=1$,
\[
  h(a/b) = \log \max\{ \abs{a}, \abs{b} \}.
\]
For more properties about the logarithmic Weil height, as well as for comparison with other notions of height, 
we refer the reader to \cite[Chapter 3]{Miw}.

\subsubsection{Landau notation}

Let $(a_n)_{n\geq 0}$ be a sequence of non-negative real numbers and $(b_n)_{n\geq 0}$ be a 
sequence of, eventually positive, real numbers. As usual, the notation $a_n\in O(b_n)$ means that there 
exists a positive number $c$ such that $a_n<cb_n$ for every sufficiently large positive integer $n$, while the 
notation $a_n\in o(b_n)$ means that $a_n/b_n$ tends to zero as $n$ tends to infinity. 
Moreover, sticking to the usual practice in number theory, we write $a_n\in \Omega(b_n)$ when 
$a_n\not\in o(b_n)$, that is, when there exists a positive number $c$ such that $a_n>c b_n$ for 
infinitely many positive integers $n$. We also write $a_n \in O\cap \Omega(b_n)$ when both 
$a_n\in O(b_n)$ and $a_n\in \Omega(b_n)$. 

\medskip

We are now ready to state our first main result. 

\begin{theorem}[Height Gap Theorem] \label{thm: hgt}
  Let $f(z) = \sum_{n=0}^\infty a_n z^n \in \power{\Qbar}$ be an $M$-function. 
  Then one of the following properties holds.

  \begin{enumerate}
  \item\label{tms:generic} $h(a_n) \in O\cap\Omega(n)$.
  
  \item\label{tms:hn} $h(a_n) \in O\cap\Omega(\log^2n)$.
   
  \item\label{tms:hlog2} $h(a_n) \in O\cap\Omega(\log n)$.
  
  \item\label{tms:hlog} $h(a_n)  \in O\cap\Omega(\log\log n)$.
  
  \item\label{tms:hloglog} $h(a_n) \in O(1)$.
  \end{enumerate}
\end{theorem}

\Cref{thm: hgt} implies that the coefficients of an $M$-function can only exhibit certain specific growth behaviors. 
For instance, as $h(a_n) \in o(n)$ forces $h(a_n) \in O(\log^2n)$, there cannot be such a power series with 
$h(a_n) \sim \log^3n$. Thus, there are \emph{gaps} in the possible growths. 
Let us make few comments on \cref{thm: hgt}. 

\begin{itemize}
\item[$\bullet$] In Section \ref{sec: zoo}, we provide the reader 
with examples for each of the five growth classes, thereby showing that all of them occur.

\medskip

\item[$\bullet$] There is no chance in general of replacing lower bounds of the type $\Omega$ by stronger ones. 
For instance, the $2$-Mahler function $\sum_{n=0}^{\infty} 2^nz^{2^n}$ belongs to class (3), but 
most of its coefficients  vanish.

\medskip

\item[$\bullet$] An $M$-function $f(z)$ can be uniquely specified by the finite data consisting of a $k$-Mahler 
equation it satisfies and sufficiently many initial coefficients of the power series. Assuming the knowledge of 
such data, we will show that it is decidable which of the five growth classes in \cref{thm: hgt} 
the function $f(z)$ falls into. 
This is Theorem \ref{t:decidability}. 

\end{itemize}

\subsection{Height and structural properties of $M$-functions}

We already alluded to the fact that inside the ring of $k$-Mahler functions two subsets are usually distinguished,  
leading to the following hierarchy: 
$$
\{\mbox{$k$-automatic functions}\}\subsetneq \{\mbox{$k$-regular functions}\}
\subsetneq \{\mbox{$k$-Mahler functions}\} \,.
$$
We refer the reader to \cite{allouche-shallit03} and Section \ref{sec:prelim} for precise definitions and more details  
about automatic and regular power series. 
The following result shows that each of these two subsets can be characterized within the $k$-Mahler functions by their coefficient growth, using the refined hierarchy provided by Theorem \ref{thm: hgt}. 

\begin{theorem}\label{thm: aut-reg}
Let $f(z)=\sum_{n=0}^\infty a_n z^n\in\power{\Qbar}$ be a $k$-Mahler function. Then the two following properties hold. 
\begin{enumerate}
      \renewcommand{\theenumi}{\rm{(\alph{enumi})}}
    \renewcommand{\labelenumi}{\theenumi}

  \item \label{aut-reg:aut} $f(z)$ is $k$-automatic if and only if $h(a_n) \in O(1)$, that is, if and only if the sequence $a_n$ 
  takes values in a finite set. 
  
    \item \label{aut-reg:reg} $f(z)$ is $k$-regular if and only if $h(a_n) \in O(\log n)$. 
  \end{enumerate}
\end{theorem}

Case \ref{aut-reg:aut} of \cref{thm: aut-reg} extends to arbitrary ground fields of characteristic zero (see \cref{c:finite-automatic}). 
This generalizes the well-known fact that $k$-regular sequences taking only finitely many values are 
$k$-automatic \cite[Theorem 16.1.5]{allouche-shallit03}.


In fact, in proving \cref{thm: hgt}, we will show that each of the five growth classes 
corresponds to natural structural properties of the 
$k$-Mahler equation, respectively, the coefficient series. The corresponding results are stated in 
\cref{t-main:hn,t-main:hlog2,t-main:hlog,t-main:hloglog}.  Theorem \ref{thm: aut-reg} 
above provides only a sample.    
In order to get such structural results, we reinforce the importance of measuring  
the size of coefficients by their height and not only by their modulus. 
For instance, all the following three Mahler functions 
$$
\prod_{n=0}^{\infty}(1-z^{2^n}) \;\; , \;\; \sum_{n= 0}^{\infty} 2^{-n}z^{2^n} \;\; , \;\;  
\frac{1}{1-z/2} \cdot \prod_{n=0}^{\infty}(1-z^{2^n})  \;\; 
$$
have 
bounded rational coefficients, so we cannot distinguish them through the growth 
of their coefficients. This is a deficiency, for the first one is automatic, 
the second one is regular but not automatic, 
and the third one is not regular. 
However, their coefficients have different height growth behaviors 
and they can be distinguished by 
Theorem \ref{thm: hgt}. They belong respectively to classes  
(1), (3), and (5).

The decidability of growth properties of $k$-regular sequences with respect to the usual (archimedean) absolute value has recently been considered by Krenn and Shallit \cite{krenn-shallit20}.
In contrast to our results, many of these properties become undecidable.
For instance, it is undecidable whether the sequence of coefficients is bounded \cite[Theorem D]{krenn-shallit20}.

\subsection*{Outline} This article is organized as follows. 
Section \ref{sec: zoo} provides the reader with  a collection of examples, showing that 
each of the five growth classes in the height gap theorem actually occurs.  
In Sections \ref{sec:prelim} and \ref{sec:back-mahler}, we give some background about Mahler equations, 
automatic and regular power series, and Mahler's method.  \cref{sec:generic,sec:hn,sec:hlog2,sec:hlog,sec:hloglog} 
are then devoted to the proof of Theorem \ref{thm: hgt}. In fact, we prove the much more precise 
\cref{t-main:hn,t-main:hlog2,t-main:hlog,t-main:hloglog}.
In Section \ref{sec: becker}, we discuss how our main results imply Becker's conjecture.
In Section \ref{sec: automatic}, we characterize those $k$-Mahler functions which are automatic over 
an arbitrary ground field of characteristic zero. In the final Section \ref{sec: decidability}, we deal with the question of decidability in \cref{thm: hgt}.  

\subsection*{Notation}
Throughout the paper, we use the following notation.  We let $k\geq 2$ be a natural number. 
We let $\digits$ denote the alphabet $\left\{0,1,\ldots,k-1\right\}$ and $\digits^*$ denote the free 
monoid generated by $\digits$, with neutral element $\varepsilon$.  
Given a positive integer $n$, we set 
$\langle n\rangle_k:=w_r w_{r-1}\cdots w_0$ for the canonical
base-$k$ expansion of $n$ (written from most to least significant digit),
which means 
that $n=\sum_{i=0}^r w_i k^{i}$ with $w_i\in \digits$ and $w_r\not= 0$. 
Note that by convention $\langle 0\rangle_k:= \varepsilon$.  
Conversely, if $w:= w_0 \cdots w_r$  is a finite word over the alphabet $\digits$, 
we set $[w]_k:=\sum_{i=0}^r w_{r-i} k^{i}$.  
We let $\ru \subseteq \Qbar$ denote the set of all roots of unity.
For $\zeta \in \Qbar$ with $\zeta \ne 0$, observe that there exists $j > 0$ with $\zeta^{k^j} = \zeta$ if and only if 
$\zeta \in \ru$ and $\zeta$ has order coprime to $k$.
We let $\rucommon \subseteq \ru$ denote the set of roots of unity whose order is not coprime with $k$.

\section{Witnessing Examples}\label{sec: zoo}

In this section, we provide examples of Mahler functions
$f(z) = \sum_{n=0}^\infty a_n z^n \in \power{\Qbar}$
for each of the five growth classes 
occurring in the height gap theorem. We recall that a rational power series is $k$-Mahler for all 
$k\geq 2$. 

\subsection*{Examples in $O\cap \Omega(n)$}

The upper bound $h(a_n) \in O(n)$ holds for every Mahler function, by \cref{thm: hgt}.
To give examples of Mahler functions whose coefficient sequence has growth in $O \cap \Omega(n)$, it therefore suffices to find examples whose coefficient series has growth in $\Omega(n)$.

  \begin{enumerate}
    \renewcommand{\theenumi}{(\alph{enumi})}
    \renewcommand{\labelenumi}{\theenumi}

  \item The coefficients of the rational function 
  $$
  \frac{1}{1-2z} = \sum_{n=0}^\infty 2^n z^n
  $$
  have linear growth since $h(2^n) = n \log(2)$.
  
  \medskip

  \item
  Let $a$, $k \ge 2$ be integers and consider the transcendental infinite product 
  \[
     \prod_{n=0}^\infty \frac{1}{1-a z^{k^n}} = \sum_{n=0}^\infty a_n z^n.
  \]
  This power series is $k$-Mahler and $a_n$ is at least as large as the coefficient of $z^n$ in $1/(1-az)$, that is $a_n \ge a^n$.
  Hence $h(a_n) \ge n \log(a)$.
   
   \medskip
   
\item \label{ex:linear-analytic}
  The previous example has poles at $a^{-1/k^n}$ for all $n \ge 1$, but it can be refined to one that is analytic in the open unit disk of $\bC$.  
  Let $a$, $k \geq 2$ be integers and let us consider the infinite product 
  \[
    \prod_{n=0}^\infty \frac{1}{1-{a^{-1}}z^{k^n}} = \sum_{n=0}^\infty a_n z^n \in \power{\bQ}.
  \] 
  A partition of $n$ into $k$-powers is an expression 
  $n = j_1 k^{n_1} + \cdots + j_r k^{n_r}$ with $r \in \bZ_{\ge 0}$, 
  $0 \le n_1 < \cdots < n_r$, and $j_1$, $\ldots\,$,~$j_r \in \bZ_{\ge 0}$.
  Expanding the factors in the definition of the infinite product as geometric series, we see that 
  \[
    a_n = \sum_{n = j_1 k^{n_1} + \cdots + j_r k^{n_r}} a^{-(j_1 + \cdots + j_r)},
  \]
  where the sum is over partitions of $n$ into $k$-powers. 
  The partition $n = 1 + \cdots + 1 = n \cdot k^{0}$ gives a summand $a^{-n}$, and 
  for all other summands $j_1+\cdots + j_r < n$.
  Let $p$ be a prime divisor of $a$.
  Then $\abs{a_n}_p \geq p^{n}$, 
  and therefore $h(a_n)  \ge n \log p$. 
\end{enumerate}

\subsection*{Examples in $O\cap \Omega(\log^2 n)$}\label{t:k-partition} 

The following example is typical of the Mahler functions in the class $O\cap \Omega(\log^2 n)$. It 
will play a prominent role in Section \ref{sec:hlog2}.  

 \begin{enumerate}
     \renewcommand{\theenumi}{(\alph{enumi})}
    \renewcommand{\labelenumi}{\theenumi}
    \setcounter{enumi}{3}
\item \label{ex:cyclo}
  Let $k \ge 2$ be an integer and consider the infinite product of cyclotomic polynomials 
 $$
  \prod_{n=0}^\infty \frac{1}{1 - z^{k^n}} = \sum_{n=0}^\infty a_n z^n\,.
 $$
 As in \ref{ex:linear-analytic} above, we see that the integer $a_n$ is equal to the number of partitions of $n$ into $k$-powers.   
 The asymptotics of the coefficient sequence $a_n$ were first studied by Mahler \cite{mahler40}, who proved that 
   \[
    \log a_n \sim \frac{\log^2n}{2 \log k}\,\cdot
  \]
These results of Mahler have been refined and generalized by de Bruijn \cite{debruijn48} 
and most recently by Dumas and Flajolet \cite{dumas-flajolet96}.

\medskip

 \item Multiplying the infinite product in \ref{ex:cyclo} by any non-zero $k$-regular power series with positive coefficients  
 provides a transcendental $k$-Mahler function with the required growth behavior. 
 In fact, \cref{t:dumas-structure} shows that examples in the class $O\cap \Omega(\log^2 n)$ are essentially all of that type.  

\end{enumerate}

\subsection*{Examples in $O\cap \Omega(\log n)$} 

For every regular sequence $(a_n)_{n \ge 0}$, the generating series 
$\sum_{n=0}^\infty a_n z^n$ is an $M$-function,  
and examples for which $h(a_n) \in O\cap \Omega(\log n)$ abound. 
We give some examples and refer the reader to \cite[Chapter 16.5]{allouche-shallit03} and 
\cite{allouche-shallit92,allouche-shallit03b} for more.

  \begin{enumerate}
  
  \item[(f)] The rational power series 
  $$
 \frac{z}{(1-z)^2}= \sum_{n=0}^\infty n z^n \,,
 $$
 falls into this class, since $h(n) = \log(n)$.
   More generally, if $p(z)$ is a non-constant polynomial with integer coefficients, then 
   $\sum_{n=0}^\infty p(n) z^n$ 
   is a rational function  with the required growth behavior.  
  
  \medskip
  
\item[(g)]  Let $\sv_p(n)$ denote the $p$-adic valuation of the natural number $n$, where $p$ is a prime number.
  The power series
  \[
    \sum_{n=0}^\infty \sv_p(n!) z^n,
  \]
  is $p$-regular \cite[Example 8]{allouche-shallit92}. 
  Moreover, by Legendre's formula $\sv_p(n!)\sim n/(p-1)$.

  \medskip
 
\item[(h)]
  Let $\ell_n$ denote the number of positive integers at most equal to $n$ that can be written as sum of three squares.
  The power series $\sum_{n=0}^\infty \ell_n z^n$ is $2$-regular \cite[Example 16.5.2]{allouche-shallit03}. 
  Since every integer not of the form $4^a(8b+7)$ can be written as a sum of three squares, the sequence has the required growth behavior.  
  
  \medskip
  
  \item[(i)] Any linear representation $(u,\mu,v)$ on the alphabet $\digits$ gives rise to a $k$-regular sequence 
  (see Definition \ref{def: lr}).
    From our results, we will see that whenever there exists a word $w \in \digits^*$ such that the matrix 
    $\mu(w)$ has an eigenvalue that is neither $0$ nor a root of unity, then the sequence associated with this linear representation has the required growth behavior. 
  \end{enumerate}

  \subsection*{Examples in $O\cap \Omega(\log\log n)$}

  $M$-functions whose coefficients have growth in $O\cap \Omega(\log\log n)$ can again be found by looking at generating functions of suitable $k$-regular sequences.

  \begin{enumerate}
 
  \item[(j)] The power series
    \[
      \displaystyle\sum_{n=1}^{\infty}(1 + \lfloor \log_2 n\rfloor)z^n
     \] is $2$-regular \cite[Example 11]{allouche-shallit92}, and clearly has the required growth behavior.   
  
  \medskip
 
\item[(k)]
  Let $s_n$ denote the sum of digits in the base-$k$ expansion of $n$.
  Then clearly $(s_n)_{n\geq 0}$ is $k$-regular, and therefore the power series $\sum_{n=0}^{\infty}s_nz^n$ is $k$-Mahler.
  Since $s_n \in O( \log n)$, it holds that $h(s_n) \in O(\log\log n)$.
  Moreover, for $e \ge 0$ and $n=k^e-1$ we have 
  $s_n = (k-1)e \sim (k-1)\log_k n$. Hence $s_n$ has he required growth behavior.

 \end{enumerate}

\subsection*{Examples in $O(1)$} By Theorem \ref{thm: aut-reg}, this class of $M$-functions  
corresponds exactly to generating series of automatic sequences. 
We refer the reader to the monograph \cite{allouche-shallit03} for numerous examples, including the 
generating series of the Thue-Morse sequence, the Rudin-Shapiro sequence, the Baum-Sweet sequence, 
and the paperfolding sequence, to name a few.

\section{Preliminaries} \label{sec:prelim}

Throughout this section, we let $K$ be a field. We will later restrict ourselves to $K=\Qbar$.
We recall $k$-Mahler, $k$-automatic, $k$-regular, and  $k$-Becker power series 
and their relation to each other.

\subsection{Mahler functions, equations, and systems}

Let us recall that a power series $f(z)\in \power{K}$ is a $k$-Mahler function if it satisfies an equation of the 
form  \eqref{eq: mahler1}, that is if there exist a non-negative integer $d$ and polynomials 
$p_0(z),\ldots,p_d(z) \in K[z]$, not all zero, such that
  \[
    p_0(z)f(z)+p_1(z)f(z^k)+\cdots +p_d(z)f(z^{k^d})=0\,.
  \]
It can be shown that every Mahler function satisfies such a functional equation with  $p_0(z)p_d(z) \ne 0$  
and  $p_0(z)$, $\ldots\,$,~$p_d(z)$ coprime \cite[Lemma 4.1]{adamczewski-bell17}.  
As we will only be interested in the asymptotic behavior of the coefficients, the following lemma allows a further simplification of the Mahler equation.
\begin{lemma} \label{l:kill-0}
  Suppose $f(z)= \sum_{n=0}^\infty a_n z^n \in \power{K}$ satisfies a Mahler equation
  \begin{equation} \label{eq:mahler}
    p_0(z) f(z) = p_1(z) f(z^k) + \cdots + p_d(z) f(z^{k^d})
  \end{equation}
  with $p_0(z) p_d(z) \ne 0$ and $p_0(z),\ldots,p_d(z)$ coprime.
  Then there exists $n_0\ge 0$ such that $a_{n_0} \ne 0$ and $f_0(z)\coloneqq \sum_{n=0}^\infty a_{n+n_0} z^n$ satisfies
  a $k$-Mahler equation
  \[
    q_0(z) f_0(z) = q_1(z) f_0(z^k) + \cdots + q_{d+1}(z) f_0(z^{k^{d+1}})
  \]
  with polynomials $q_0(z),\ldots,q_{d+1}(z)$ satisfying the following conditions.
  \begin{propenumerate}
  \item One has $q_0(0)=1$.
  \item If $\lambda \in K\setminus\{0\}$, then $p_0(\lambda) = 0$ implies $q_0(\lambda) = 0$.
  \item If $\zeta \in K\setminus\{0\}$  with $p_0(\zeta) = 0$ and $\zeta^k = \zeta$, then $q_i(\zeta) \ne 0$ for some $i \in \{1,\ldots,d+1\}$.
  \end{propenumerate}
  Moreover, if $f(z)$ has at least two non-zero coefficients, then $f_0(z)$ is non-constant.
\end{lemma}

\begin{proof}
  By \cite[Lemma 6.1]{adamczewski-bell17}; the final statement requires an inspection of the proof.
\end{proof}

We also need the following fact, a more general version of which is, for instance, proved 
in \cite[Proposition 8.1]{adamczewski-bell17}.

\begin{lemma} \label{l:mahler-power}
  If $f(z) \in \power{K}$ is $k$-Mahler and $e$ is a positive integer, then $f(z)$ is also $k^e$-Mahler.
\end{lemma}

\subsubsection{Linear Mahler systems} 
A power series $f(z)\in \power{K}$ is $k$-Mahler if and only if it satisfies a \emph{linear $k$-Mahler system}.
  That is, there exist $f_1(z):=f(z),\ldots,f_d(z) \in \power{K}$ and $A(z) \in \GL_d(K(z))$ such that
    \begin{equation}
      \label{eq:mahler-lin}
      \begin{pmatrix}
        f_1(z)\\
        \vdots \\
        f_d(z)
      \end{pmatrix}
      = A(z)
      \begin{pmatrix}
        f_1(z^k)\\
        \vdots \\
        f_d(z^k)
      \end{pmatrix}.
    \end{equation}
    Indeed, given $f(z)$ satisfying a $k$-Mahler equation $f(z) = r_1(z) f(z^k) + \cdots + r_d(z) f(z^{k^d})$ with 
    $r_1(z),\ldots\,r_d(z) \in K(z)$ and $r_d(z) \ne 0$, the vector
    \[
      \left( f(z),\ldots,f(z^{k^{d-1}})\right)^T
    \]
    satisfies an equation of the form \eqref{eq:mahler-lin} with $A(z)$ a companion matrix. 
    Conversely, iterating an equation of the form \eqref{eq:mahler-lin}, and using the invertibility of $A(z)$, 
    it follows that each $f_i(z^{k^j})$ is contained in the finite-dimensional $K(z)$-vector space spanned by 
    $f_1(z),\ldots,f_d(z)$. Hence the power series $f_1(z^{k^j})$, $j\geq 0$, are linearly dependent over $K[z]$.

\subsubsection{Analytic properties} Let us assume that $K=\Qbar$. If $f(z) \in \power{\Qbar}$ is a 
$k$-Mahler function, then there exists a number 
field $K$ with $f(z) \in \power{K}$. This is so because all sufficiently high coefficients of $f(z)$ are determined recursively by lower ones (see \cite[Chapitre 3.2.2]{dumas93} or \cite{adamczewski-faverjon18}). Let $\sv$ be a place of $K$ and 
$\abs{\cdot}_\sv$ be an absolute value associated with $\sv$.
    We let $K_\sv$ denote the completion of $K$ with respect to the absolute value $\vert\cdot\vert_\sv$.
    We also let $C_\sv$ denote the completion of the algebraic closure of $K_\sv$ and $\overline{K}$ the algebraic closure of $K$ in $C_\sv$.
    Recall that $C_\sv$ is both algebraically closed and complete. 
    The power series $f(z)$ is analytic in a neighborhood of $0$ in $C_\sv$ (see, for instance, \cite[Chapitre 3.3]{dumas93}).
    The Mahler equation then implies that $f(z)$ is meromorphic in the open unit disk $B_{\abs{\cdot}_\sv}(0,1)$ in $C_\sv$.
\subsection{Automatic and 
regular power series} \label{sec:back-regular}

We recall the notion of $k$-automatic and $k$-regular sequences.
For more background see Allouche and Shallit \cite{allouche-shallit03} or 
Berstel and Reutenauer \cite[Chapter 5]{berstel-reutenauer11}.

A sequence $(a_n)_{n \ge 0}$ is \emph{$k$-automatic} if there exists a finite automaton that, 
given as input the base $k$ representation of $n$, reaches an output state labeled by $a_n$. 
Equivalently, the sequence $(a_n)_{n\ge 0}$ is $k$-automatic if and only if its $k$-kernel is a finite set. 

\begin{definition}
  Let ${\bf a}:=(a_n)_{n \ge 0}$ be a sequence with values in a set $S$.
  The \emph{$k$-kernel} of $\bf a$ is
  \[
    \big\{\, \left(a_{k^en + r}\right)_{n \ge 0} : e \in \bZ_{\ge 0},\,  0 \le r \le k^e - 1 \,\big\}.
  \]
\end{definition}

Let us now restrict to sequences taking values in the field $K$.
Then a sequence $(a_n)_{n \ge 0}$ is said to be \emph{$k$-regular} if its $k$-kernel is contained in a finite-dimensional vector subspace of the $K$-vector space of all $K$-valued sequences. Obviously $k$-automatic sequences are $k$-regular.
A $k$-regular sequence is $k$-automatic if and only if it takes only finitely many values 
\cite[Theorem 16.1.5]{allouche-shallit03}. 
There are several other characterizations of $k$-regular sequences 
\cite[Theorems 16.1.3 and 16.2.3]{allouche-shallit03}. 
We recall one characterization that is relevant for our purpose.

\begin{definition}\label{def: lr} 
A \emph{linear representation} on the alphabet $\digits$ is a triple $(u,\mu,v)$ where $u \in K^{1 \times d}$, 
$v \in K^{d \times 1}$, and $\mu \colon \digits^* \to K^{d \times d}$ is a monoid homomorphism 
($d \in \bZ_{\ge 0}$). The linear representation is \emph{minimal} if the dimension $d$ is minimal 
amongst all $d' \ge 0$ and $d'$-dimensional linear representations $(u',\mu',v')$ such that
 $u\mu(w)v = u'\mu'(w)v'$ for all $w \in \digits^*$. Equivalently, $u \mu(\digits^*)$ spans 
 $K^{1\times d}$ and $\mu(\digits^*)v$ spans $K^{d \times 1}$. 
\end{definition}

\begin{theorem} \label{t:regular-linrep}
  Let $(a_n)_{n \ge 0}$ be a sequence taking values in $K$.
  The following statements are equivalent.
  \begin{equivenumerate}
  \item\label{regular:regular} The sequence $(a_n)_{n \ge 0}$ is $k$-regular.
  \item\label{regular:linrep} There exists a \textup{(}minimal\textup{)} linear representation $(u, \mu, v)$ 
  on the alphabet $\digits$ such that $a_{[w]_k} = u \mu(w) v$ for all words $w \in \digits^*$.
  \end{equivenumerate}
\end{theorem}

\begin{proof}
  The result is proved in \cite[Theorem 16.2.3]{allouche-shallit03}.
\end{proof}

A power series $f(z) = \sum_{n=0}^\infty a_n z^n \in \power{K}$ is said to be \emph{$k$-automatic}, respectively  
\emph{$k$-regular} if the sequence $(a_n)_{n \ge 0}$ is $k$-automatic, respectively $k$-regular. 
\subsection{Becker power series}\label{subsec: becker}

The connection between $k$-regular sequences in the sense of Allouche and Shallit and 
coefficients of $k$-Mahler power series was studied by Becker, who proved that $k$-regular power series are $k$-Mahler \cite[Theorem 1]{becker94}. 
He also showed that the converse is false in general: a $k$-Mahler power series need \emph{not} be $k$-regular \cite[Proposition 1]{becker94}.
However, he did obtain a partial converse.
This motivates the next definition.

\begin{definition} \label{d:becker}
  A power series $f(z) \in \power{K}$ is a \emph{$k$-Becker function} (or, in short, \emph{$k$-Becker}) if there exist 
  a positive integer $d$ and polynomials $p_1(z),\ldots,p_d(z) \in K[z]$, not all zero, such that
  \[
    f(z) = p_1(z) f(z^k) + \cdots + p_d(z) f(z^{k^d}) \,.
  \]
\end{definition}

\begin{theorem}[{\cite[Theorem 2]{becker94}}] \label{t:becker}
  If $f(z) = \sum_{n=0}^\infty a_n z^n \in \power{\Qbar}$ is a $k$-Becker power series, then it is $k$-regular.
\end{theorem}

In view of these results, one may also ask for a precise characterization of $k$-regular power series 
in terms of $k$-Becker power series. 
This gives rise to a conjecture of Becker, recently settled in \cite{bell-chyzak-coons-dumas18}, 
and discussed in Section \ref{sec: becker}. 
For Mahler functions, there exists the following useful decomposition due to Dumas.

\begin{theorem}[{\cite[Th\'eor\`eme 31, p.153]{dumas93}}] \label{t:dumas-structure}
  Let $f(z) \in \power{K}$ be $k$-Mahler satisfying an equation
  \[
   p_0(z) f(z) + p_1(z) f(z^k) + \cdots + p_d(z) f(z^{k^d})=0 
  \]
  with $p_0(z),\ldots,p_d(z) \in K[z]$ and $p_0(0)=1$. 
  Then
  \[
    f(z) = \frac{g(z)}{\prod_{i=0}^\infty p_0(z^{k^i})} \,,
  \]
  where  $g(z) \in \power{K}$ is a  $k$-Becker power series.
\end{theorem}

\subsection{The Mahler denominator}\label{subsec: kden}

As is already hinted at by Becker's result, the polynomial $p_0(z)$ in a Mahler equation \eqref{eq: mahler1} 
will play a prominent role in our arguments. This prompts the following definition.

\begin{definition} \label{d:mahler-denominator}
  Let $f(z) \in \power{K}$ be a $k$-Mahler power series, and let
  \[
    \mathfrak I = \big\{\, p(z) \in K[z] : p(z)f(z) \in \sum_{i=1}^\infty K[z] f(z^{k^i}) \,\big\}\,.
  \]
  The \emph{$k$-Mahler denominator} of $f(z)$ is the unique generator $\mathfrak d(z) \in K[z]$ 
  of the ideal $\mathfrak I$, with the lowest non-zero coefficient of $\mathfrak d(z)$ being $1$.
\end{definition}

Since $K[z]$ is a principal ideal domain, there indeed exists such a generator.
Observe that $f(z)$ is $k$-Becker if and only if $\mathfrak d(z)\equiv1$. 
It is tempting to hope that the $k$-Mahler denominator is equal to the polynomial $p_0(z)$ in the minimal $k$-Mahler equation, that is the equation
$$
p_0(z) f(z) + p_1(z) f(z^k) + \cdots + p_d(z) f(z^{k^d})=0
$$ 
with $p_0(z)p_d(z) \ne 0$, minimal $d$, and coprime $p_0(z),\ldots,p_d(z)$.
While this is often the case, in general this is \emph{not} so. See \cref{exm: denominator} for a counterexample. 
By definition $\mathfrak d(z)$ divides $p_0(z)$.
It is tempting to hope that, to determine the types of roots of $\mathfrak d(z)$, 
it suffices to consider those of $p_0(z)$.
Unfortunately, this hope is also thwarted by the following example.
\begin{example} \label{exm: denominator}
  The equation 
  \[
  (z-1/2)f(z) - (z-1/8)(z^3-1/2)f(z^3) =0
  \]
  has only one non-zero solution (up to a scalar) and is minimal with respect to this solution. 
  However, this solution is $k$-regular because
  \[
  f(z)= (z-1/8)(z^2+1/2z+1/4)(z^9-1/2)f(z^9)\,.
  \]
  The expected pole at $1/2$ disappears after one iteration of the equation.
\end{example}

We will see with Theorems \ref{t-main:hn} and \ref{t-main:hlog2} that locating the roots of 
the $k$-Mahler denominator provides a  
characterization of those $k$-Mahler functions with $h(a_n)\in O(\log^2 n)$ and with $h(a_n)\in O(\log n)$. 
However, given a Mahler function with $h(a_n)\in O(\log n)$, 
its Mahler denominator is irrelevant in determining whether 
$h(a_n)\in O(\log\log n)$ or $h(a_n)\in O(1)$.  
For instance, all the three following $2$-regular functions 
$$
\prod_{n=0}^{\infty}(1-z^{2^n}) \;\; , \;\; \sum_{n= 0}^{\infty} b_nz^n \;\; , \;\;  
\left(\frac{1}{1-z} \right)^2\;\;,
$$ 
where we let $b_n$ denote the number of $0$'s in the binary expansion of $n$ (with $b_0=1$), 
have a trivial Mahler denominator ({\it i.e.}, $\mathfrak d(z) =1$).  
However, their coefficients have height  
in $O(1)$, $O\cap \Omega(\log\log n)$, and  
$O\cap \Omega(\log n)$, respectively. 
\section{Background about Mahler's method} \label{sec:back-mahler}

Let us consider a linear $k$-Mahler system:

\begin{equation}\label{eq: systeme}
  \begin{pmatrix}
     f_1(z) \\
     \vdots \\
     f_d(z)
   \end{pmatrix}
   = A(z)
   \begin{pmatrix}
     f_1(z^k) \\
     \vdots \\
     f_d(z^k)
  \end{pmatrix}
\end{equation}
 where $A(z)$ is a matrix in $\GL_d(\overline{\mathbb Q}(z))$ and $f_1(z),\ldots,f_d(z) \in \power{\Qbar}$.
 There exists a number field $K$ such that the $f_i(z)$'s belong to $\power{K}$ and 
 $A(z)\in \GL_d(K(z))$. Let $\sv$ be a place of $K$ and $\abs{\cdot}_\sv$ be an absolute 
 value associated with $\sv$. As before, we let $K_\sv$ denote the completion of $K$ with 
 respect to the absolute value $\abs{\cdot}_\sv$. We also let $C_\sv$ denote the completion 
 of the algebraic closure of $K_\sv$ and $\overline{K}$ the algebraic closure of $K$ in $C_\sv$. 

\begin{definition}A point $\alpha\in C_\sv$ is called \emph{singular} with respect to 
  \eqref{eq: systeme} if there exists a non-negative integer $n$ such that $\alpha^{k^n}$ 
  is a pole of one of the coefficients of the matrix $A(z)$ or of the matrix $A^{-1}(z)$. 
  We say that $\alpha$ is \emph{regular} otherwise, that is, $\alpha$ is regular if both 
  $A(\alpha^{k^n})$ and $A^{-1}(\alpha^{k^n})$ are well-defined for every non-negative integer $n$.
\end{definition}

We recall that the power series $f_1(z),\ldots,f_d(z)$ are meromorphic in the open unit disc of 
$C_\sv$ and analytic in some neighborhood of the origin. Moreover, if $\alpha$ is a regular 
point such that $\abs{\alpha}_\sv < 1$, then the functions $f_1(z),\ldots,f_d(z)$ 
are well-defined at $\alpha$. We also recall that given a field $K$, 
and elements $a_1,\ldots,a_m$ in some field extension of $K$,  
the notation ${\rm tr.deg}_{K}(a_1,\ldots,a_m)$ stands for the transcendence degree over 
$K$ of the field extension $K(a_1,\ldots,a_m)$. 

\begin{theorem}
\label{thm: nishioka}
Let  $f_1(z),\ldots,f_d(z) \in \power{K}$ be related by a Mahler system of the form \eqref{eq: systeme}.
Let $\alpha\in \overline{K}$, $0<\abs{\alpha}_\sv <1$ be a regular point with respect to this system.
Then
\[
\trdeg_{\overline{K}} (f_1(\alpha),\ldots,f_d(\alpha)) = \trdeg_{\overline{K}(z)} (f_1(z),\ldots,f_d(z))\, .
\]
\end{theorem}
 
In the case where $\abs{\cdot}_\sv$ is the usual absolute value on $\mathbb C$, this classical 
result is due to Nishioka \cite{nishioka90}. The proof of Nishioka is based on some techniques 
from commutative algebra introduced in the framework of algebraic independence by Nesterenko 
in the late Seventies. Recently, Fernandes \cite{fernandes18} observed that \cref{thm: nishioka} 
can also be deduced from a general algebraic independence criterion due to Philippon 
\cite{philippon86,philippon92}. This allows her to extend Nishioka's theorem in the framework 
of function fields of positive characteristic. Using the fact that the criteria obtained by Philippon 
also apply to any absolute value associated with a place of a number field (see for instance 
Theorem 2.11 in \cite{philippon86}), we can argue exactly as in the proof of Theorem 1.3 
of \cite{fernandes18} to prove \cref{thm: nishioka}.

\begin{theorem}
\label{thm: lifting}
Let $f_1(z),\ldots,f_d(z)\in \power{K}$ be related by a Mahler system of the form \eqref{eq: systeme}.
Let $\alpha\in\overline{K}$, $0<\abs{\alpha}_\sv <1$ be a regular point for this system.
Then for all homogeneous polynomials $P(X_1,\ldots,X_d)\in \overline{K}[X_1,\ldots,X_d]$ such that
\[
  P(f_1(\alpha),\ldots,f_d(\alpha))=0\, ,
\]
there exists $Q(z,X_1,\ldots,X_d)\in \overline{K}[z,X_1,\ldots,X_d]$, homogeneous in $X_1$, $\ldots\,$,~$X_d$, 
such that
\[
  Q(z,f_1(z),\ldots,f_d(z))=0
\]
and
\[
  Q(\alpha,X_1,\ldots,X_d)=P(X_1,\ldots,X_i).
\]
\end{theorem}

\begin{proof}
  In the case where $\abs{\cdot}_\sv$ is the usual absolute value on $\mathbb C$, this result is 
  due to Adamczewski and Faverjon in \cite[Theorem 1.4]{adamczewski-faverjon17}. 
  It is obtained as a consequence of the main result of Philippon in \cite{philippon15}, 
  which itself is based on Nishioka's theorem. The strategy to deduce this result from Nishioka's theorem 
  is detailed in \cite{adamczewski-faverjon17}, see Proposition 3.1.
  The arguments are based on basic facts from commutative algebra that also apply to 
  our more general framework. The two main ingredients that we have to be careful about 
  are the following ones.
 
\begin{enumerate}[label=(\roman*)]
  \item \label{af-fact:1}
    A result by Krull saying that if $\mathfrak{p}$ is a homogeneous ideal in $K[z,X_0,\ldots,X_d]$ 
    that is absolutely prime, then for all but finitely many $\alpha\in K$, the ideal $\ev_\alpha(\mathfrak p)$ 
    is a prime ideal of $K[X_0,\ldots,X_d]$. Here, we let $\ev_{\alpha}\colon K[z]\mapsto K$ denote 
    the evaluation map at $z=\alpha$. See \cite{krull48}.

  \item \label{af-fact:2}
    The fact that the field extension $L\coloneqq\overline{K}(z)(f_1(z),\ldots,f_d(z))$ is regular, 
    which means that an element of $L$ is algebraic over $\overline{K}(z)$ if and only if it 
    belongs to $\overline{K}(z)$. 
 \end{enumerate}

  We can use \ref*{af-fact:1} in our framework for Krull proved his result for any base field $K$.
  To prove that \ref*{af-fact:2} also holds true in our framework, we need to know that a $k$-Mahler 
  function in $\power{\overline{K}}$ is either rational or transcendental over $\overline{K}(z)$.
  There are several proofs for this result. For instance, Theorem 5.1.7 in \cite{nishioka96} provides 
  a proof in the case where $K$ is any field of characteristic $0$. Then we can argue exactly as in 
  the proof of Lemma 3.2 in \cite{adamczewski-faverjon17} to deduce that the field extension 
  $\overline{K}(z)(f_1(z),\ldots,f_d(z))$ is regular.
\end{proof}

As a corollary of \cref{thm: lifting}, we deduce the following result. 

\begin{corollary}
  \label{cor: lifting}
  Let $f_1(z),\ldots,f_d(z)\in \power{K}$ be related by a Mahler system of the form \eqref{eq: systeme}.
  Let us assume that $f_1(z),\ldots,f_d(z)$ are linearly independent over $\overline{K}(z)$. 
  Then there exists a real number $r$,  $0 < r < 1$, such that for every $\alpha\in\overline{K}$ with $0<\abs{\alpha}_\sv <r$, 
  the numbers $f_1(\alpha),\ldots,f_d(\alpha)$ are well-defined and linearly independent over $\overline{K}$. 
\end{corollary}

\begin{proof}
  We first observe that if $r$ is small enough, then $\alpha$ is a regular point 
  with respect to \eqref{eq: systeme} and the numbers $f_1(\alpha),\ldots,f_d(\alpha)$ are 
  thus well-defined. Then the result follows directly from \cref{thm: lifting}.  
\end{proof}

\section{Generic upper bound} \label{sec:generic}

To prove \ref{tms:generic} of \cref{thm: hgt}, giving a general upper bound on $h(a_n)$ 
for a Mahler function $f(z) = \sum_{n=0}^\infty a_n z^n \in \power{\Qbar}$, we use a classical recursion for the 
sequence $(a_n)_{n\ge 0}$ that is deduced from the Mahler equation. 
Since this is somewhat lengthy and the proof of the upper bound for $h(a_n)$ in case \ref{tms:hn} of 
\cref{thm: hgt}, where we assume $h(a_n) \in o(n)$, works similarly, we establish both 
these bounds at the same time.

  We need the following lemma. For archimedean absolute values, a more general result for finitely 
  generated semigroups of matrices can be found in \cite{bell05}.

\begin{lemma} \label{l:opnorm-growth}
  Let $d$ be a positive integer. Let $\abs{\cdot}$ be an absolute value on $\Qbar$, 
  and let $\norm{\cdot}$ be an operator norm on $\Qbar^{d \times d}$ with respect to $\abs{\cdot}$.
  Let $A \in \Qbar^{d\times d}$ be a matrix such that $\abs{\lambda} \le 1$ for every eigenvalue 
  $\lambda$ of $A$.
  Then
  \[
    \norm{A^n} \in
    \begin{cases}
      O(n^{d-1}) & \text{if $\abs{\cdot}$ is archimedean,} \\
      O(1) & \text{if $\abs{\cdot}$ is non-archimedean.}
    \end{cases}
  \]
\end{lemma}

\begin{proof}
  It suffices to show the claim for a Jordan block $\lambda + N \in \Qbar^{s \times s}$ 
  where $s \le d$, where $\abs{\lambda} \le 1$, and where $N$ is the $s\times s$-matrix 
  with ones on the superdiagonal and zeroes everywhere else. Then $N^i$ is the matrix 
  that has ones on the $i$th superdiagonal and zeroes everywhere else, with $N^i = 0$ 
  for $i \ge s$. Thus
  \[
    (\lambda + N)^{n} =  \sum_{i=0}^{s-1} \binom{n}{i} \lambda^{n-i} N^{i} \qquad\text{for $n \ge s$.}
  \]
  Now $\norm{(\lambda + N)^n} \le C \abs[\big]{\binom{n}{s-1}}$ for some constant $C$, 
  and the claim follows.
\end{proof}

\begin{proposition} \label{p:upper-generic}
  Let $f(z) = \sum_{n=0}^\infty a_n z^n \in \power{\Qbar}$ be a $k$-Mahler function. 
  Then the following properties hold. 
  \begin{enumerate}
  \item \label{upper-generic:n} One has $h(a_n) \in O(n)$.
  \item \label{upper-generic:logn2} Suppose in addition that all roots of the $k$-Mahler denominator of $f(z)$ 
  are contained in $\{0\} \cup \ru$. Then $h(a_n) \in O(\log^2 n)$. 
  \end{enumerate}
\end{proposition}

\begin{proof}
  We may assume $f(z) \ne 0$.
  The series $f(z)$ satisfies a $k$-Mahler equation
  \[
    p_0(z) f(z) = p_1(z) f(z^k) + \cdots + p_d(z) f(z^{k^d})
  \]
  with $p_0(z)=\mathfrak d(z)$ the $k$-Mahler denominator and $d \ge 1$.
  Extend the sequence $a_n$ to rational indices by setting $a_r=0$ for all 
  $r \in \mathbb Q\setminus \bZ_{\ge 0}$.
  Let $s = \max \{\, \deg p_i(z) : i \in 0, \ldots, d \,\}$. 
  Let $p_0(z) = z^m(1 + \alpha_1 z + \cdots + \alpha_{s-m} z^{s-m})$ with $\alpha_1$, $\ldots\,$,~$\alpha_{s-m} \in \Qbar$, and, for $i \in \{1,\ldots,d\}$,  
  let $p_i(z) = \sum_{j=0}^s \beta_{i,j} z^j$ with $\beta_{i,j} \in \Qbar$.
  Comparing coefficients in the Mahler equation, we have
  \[
    a_{n-m} + \sum_{j=1}^{s-m} \alpha_j a_{n-m-j} = \sum_{i=1}^d \sum_{j=0}^s \beta_{i,j} a_{\frac{n-j}{k^i}} \qquad\text{for $n \in \bZ$.}
  \]
  Shifting the indices by $m$, we obtain
  \begin{equation}\label{eq: recurrence}
    a_{n} = \sum_{j=1}^{s-m} -\alpha_j a_{n-j} + \sum_{i=1}^d \sum_{j=0}^s \beta_{i,j} a_{\frac{n+m-j}{k^i}} 
    \qquad\text{for $n \in \bZ$.}
  \end{equation}
  If $n > m$, then $n > (n+m)/2 \ge (n+m-j)/k^i$ for $i \ge 1$ and $j \ge 0$.
  Thus, Equation~\eqref{eq: recurrence} allows the recursive computation of $a_n$ for $n > m$ from $a_0$, $\ldots\,$,~$a_m$.
  We now write this as a matrix equation.
  For $i \in \{0,\ldots,d\}$, let
  \begin{equation} \label{eq:def-ai}
    \bba_i(n) \coloneqq
    \begin{pmatrix}
      a_{n/k^i} \\
      a_{(n-1)/k^i} \\
      \vdots \\
      a_{(n-s)/k^i}
    \end{pmatrix}.
  \end{equation}

  Let $A$, $B_1$,~$\ldots\,$,~$B_d \in \Qbar^{(s+1) \times (s+1)}$ be given by
  \[
   B_i=
    \begin{pmatrix}
      \beta_{i,0} & \beta_{i,1} & \dots & \beta_{i,s} \\
      0 & 0 & \dots & 0 \\
      \vdots & \vdots & \ddots & \vdots \\
      0 & 0 & \dots & 0
    \end{pmatrix}
    \quad\text{and}\quad
    A =
    \begin{pmatrix}
      -\boldsymbol \alpha& 0  \\
      I_{s \times s} & 0_{s \times 1}
    \end{pmatrix},
  \]
  where $\boldsymbol \alpha = (\alpha_1, \ldots, \alpha_{s-m}, 0, \ldots, 0) \in \Qbar^{1 \times s}$, where $I_{s \times s}$ is the $s \times s$ identity matrix, and $0_{s \times 1}$ is the $s \times 1$ matrix containing zeroes.
  The characteristic polynomial of $A$ is $z^{s+1} + \alpha_1 z^s + \cdots + \alpha_{s-m} z^{m+1}  = z^{m+s+1} p_0(1/z) \in \Qbar[z]$.

  Now
  \[
    \bba_0(n) = A \bba_0(n-1) + \sum_{i=1}^d B_i \bba_i(n+m) \qquad\text{for $n > m$}.
  \]
  Fix $n_0 > m$.
  Recursively substituting for $\bba_0(n-j)$, for $n \ge n_0$, we get
  \[
    \bba_0(n) = A^{n-n_0} \bba_0(n_0) + \sum_{j=0}^{n-n_0-1} \sum_{i=1}^d A^{j} B_i \bba_i(n+m-j).
  \]

  The recursion formula for $(a_n)_{n\ge 0}$ implies that there is a number field $K$ containing 
  all $\alpha_i$, $\beta_{i,j}$ and $a_n$ for $n \ge 0$.
  For each place $\sv$ of $K$, let $\abs{\cdot}_\sv$ be the corresponding absolute value and 
  let $\norm{\cdot}_\sv$ be the induced maximum norm.
  We also write $\norm{\cdot}_\sv$ for the operator norm on $K^{(s+1)\times(s+1)}$.
  Let $\varepsilon_\sv(n)=n$ if $\sv$ is archimedean, and $\varepsilon_\sv(n) = 1$ if $\sv$ is non-archimedean.
  Then
  \begin{equation} \label{eq:recursion-bound}
    \norm{\bba_0(n)}_\sv \,\le\, \varepsilon_\sv(dn) \cdot \max\Big\{ \norm{A^{n-n_0}}_\sv \cdot
    \norm{\bba_0(n_0)}_\sv,\  \norm{A^j}_\sv \cdot \norm{B_i}_\sv \cdot \norm{\bba_i(n+m-j)}_\sv \,\Big\}
  \end{equation}
  where $i \in \{1,\ldots,d\}$ and $j \in \{0,\ldots,n-n_0-1\}$.
  Let $S$ be the \emph{finite} set consisting of all places $\sv$ that are archimedean or for which 
  $\norm{A}_\sv > 1$, or $\abs{a_n}_\sv > 1$ for some $n \in \{0,\ldots,n_0\}$, or $\norm{B_i}_\sv > 1$ 
  for some $i \in \{1,\ldots,d\}$. Note that, for $\sv \not \in S$, also $\norm{A^n}_\sv \le \norm{A}_\sv^n \le 1$ 
  for all $n \ge 1$. If $\sv \not \in S$, then, by induction, the bound in \eqref{eq:recursion-bound} implies 
  $\abs{a_n}_\sv \le \norm{\bba_0(n)}_\sv \le 1$ for all $n \ge n_0$.
  Therefore
  \[
    h(a_n) = \log \prod_{\sv \in S} \max\{1, \abs{a_n}_\sv\}.
  \]
  To show the claims, it suffices to obtain suitable bounds on $\abs{a_n}_\sv$ for $\sv \in S$.
  We first prove the bound in \ref*{upper-generic:n}.
  
  Let $\sv \in S$.
  We show $\abs{a_n}_\sv \le c^{n}$ for some $c \in \bR_{\ge 1}$ and all $n \ge 1$.
  First enlarge $n_0 > m$ and pick $c_0 \in \bR_{\ge 1}$ so that $dn c^{1+m/2}\le c^{n/6}$ for all $c \in \bR_{\ge c_0}$ and $n \ge n_0$.  
  Let $c \in \bR_{\ge c_0}$ be sufficiently large so that $\abs{a_n}_\sv \le c$ 
  for all $n \in \{1,\ldots, n_0\}$ and so that $\norm{B_i}_\sv \le c$ for all $i \in \{1,\ldots,d\}$.
  Enlarging $c$ further, also suppose $\norm{A}_\sv \le c^{1/3}$, so that $\norm{A^n}_\sv \le c^{n/3}$.

  We proceed by induction on $n$.
  For $1 \le n \le n_0$ the claim is true by choice of $c$.
  For $n > n_0$, the inequality~\eqref{eq:recursion-bound} gives
  \[
    \norm{\bba_0(n)}_\sv \le \varepsilon_\sv(dn) \cdot c^{(n-n_0)/3} \cdot
    \max\Big\{ \norm{\bba_0(n_0)}_\sv ,\ \norm{B_i}_\sv \cdot \norm{\bba_i(n+m-j)}_\sv \Big\},
  \]
  where $i \in \{1,\ldots,d\}$ and $j \in \{0,\ldots,n-n_0-1\}$.
  By induction hypothesis, we can use \eqref{eq:def-ai} to estimate $\norm{\bba_i(n+m-j)}_\sv \le c^{(n+m)/k} \le c^{(n+m)/2}$, and therefore 
  \[
    \norm{\bba_0(n)}_\sv \le dn \cdot c^{(n-n_0)/3} \cdot c \cdot c^{(n+m)/2} = d n c^{1+m/2} \cdot c^{-n_0/3} \cdot c^{5n/6} \le c^n.
  \]
  Thus $\abs{a_n}_\sv \le c^{n}$, as claimed. 

  \smallskip
  To show \ref*{upper-generic:logn2}, we now assume in addition that all roots of $p_0(z)$ are contained in $\{0\} \cup \ru$.
  Let $\sv \in S$.
  We show that there exists $c \in \bR_{\ge 1}$ such that $\abs{a_n}_\sv \le n^{c \log n}$ for all sufficiently large $n$.
  To this end, first note that we may choose $c_0 \in \bR_{\ge 1}$ and enlarge $n_0 > m$ so that $dc^3n^{s+1} \le n^{c \log(n)}$ for all $n \ge n_0$ and $c \ge c_0$.
  Now let $c \in \bR_{\ge c_0}$ be sufficiently large so that $\abs{a_n}_\sv \le c$ for all $n \in \{1,\ldots,n_0\}$ and $\norm{B_i}_\sv \le c$ for all $i \in \{1,\ldots,d\}$.
  By our assumption on the roots of $p_0(z)$, all the eigenvalues of $A$ are contained in $\{0\} \cup \ru$.
  Thus $\norm{A^n}_\sv \in O(n^s)$ by \cref{l:opnorm-growth}, and we can also assume 
  $\norm{A^n}_\sv \le c n^{s}$ for $n \ge 1$, enlarging $c$ if necessary.
 
  Since $k \ge 2 > 1+m/n_0$, enlarging $c$ further, we may also assume
  \[
    c (\log k  - \log(1+m/n_0))  > s + 1
  \]
  and even $dc^2 \le n_0^{c (\log k - \log(1+m/n_0)) - (s+1)}$.
  Then $dc^2 \le n^{c (\log k - \log(1+m/n)) - (s + 1)}$ for all $n \ge n_0$.

  We show $\norm{\bba_0(n)}_\sv \le n^{c \log n}$ for all $n \ge n_0$ by induction.
  Let $n \ge n_0$.
  The bound \eqref{eq:recursion-bound} gives
  \[
    \norm{\bba_0(n)}_\sv \le \varepsilon_\sv(dn) c n^{s} \cdot
    \max\Big\{ \norm{\bba_0(n_0)}_\sv,\ \norm{B_i}_\sv \cdot \norm{\bba_i(n+m-j)}_\sv \Big\},
  \]
  where $i \in \{1,\ldots,d\}$ and $j \in \{0,\ldots,n-n_0-1\}$.
  By induction hypothesis, we can use \eqref{eq:def-ai} to estimate
  \[
    \norm{\bba_i(n+m-j)}_\sv \,\le\, \max\Big\{ c, \Big(\frac{n+m}{k}\Big) ^{c \log\big( \frac{n+m}{k} \big)} \Big\}.
  \]
  With the latter bound of the maximum,
  \[
    \begin{split}
      \norm{\bba_0(n)}_\sv &\le dcn^{s+1} \cdot c \cdot \Big(\frac{n+m}{k}\Big)^{c \log\big(\tfrac{n+m}{k}\big)} \le dc^2 n^{s+1 + {c \log\big(\tfrac{n+m}{k}\big)}}  \\
      &\le dc^2 n^{s+1 + c \log(n) + c \log\big(1 + \tfrac{m}{n}\big) - c \log(k)} \le n^{c \log(n)};
    \end{split}
  \]
  in case $\norm{\bba_i(n+m-j)}_\sv \le c$ we have
  \[
    \norm{\bba_0(n)}_\sv \le dc^3 n^{s+1} \le n^{c \log(n)}. \qedhere
  \]
\end{proof}

We have thus established the general growth bound for the coefficients of a Mahler function: 
the height of the $n$th coefficient is at most linear in $n$.

\section{First gap: characterization of totally analytic Mahler functions}
\label{sec:hn}

In this section, we characterize $k$-Mahler functions 
$f(z) = \sum_{n=0}^\infty a_n z^n \in \power{\Qbar}$ with $h(a_n) \in o(n)$. 
Let $f(z)\in\power{\Qbar}$ be an $M$-function. There exists a number field $K$ such that $f(z)\in\power{K}$, 
and for every place $\sv$ of $K$, we may consider $f(z)$ as a power series over the algebraic closure $C_\sv$ 
of the completion $K_\sv$. Then $f(z)$ has a positive radius of convergence, and it is meromorphic in the open 
unit disk of $C_\sv$. Moreover, for all but finitely many places of $K$, the radius of convergence of 
$f(z)$ is equal to $1$.  Hence, an $M$-function is \emph{globally analytic}. 
We say that $f(z)$ is \emph{totally analytic} if, for every place $\sv$ of $K$, $f(z)$ 
  is analytic in the open unit disk of $C_\sv$.

\begin{theorem} \label{t-main:hn}
  Let $f(z) = \sum_{n=0}^\infty a_n z^n \in \power{\Qbar}$ be a $k$-Mahler function.
  The following statements are equivalent.
  \begin{equivenumerate}
  \item \label{hn:hon} We have $h(a_n) \in o(n)$.
   \item \label{hn:denominator} Every non-zero root of the $k$-Mahler denominator of $f(z)$ belongs to $\ru$ 
   ({\it i.e.}, is a root of unity).
  \item \label{hn:radius} The power series $f(z)$ is totally analytic.
  \item \label{hn:hOlog2} We have $h(a_n) \in O(\log^2n)$.
  \end{equivenumerate}
\end{theorem}

The crucial step here lies in showing that all roots of the $k$-Mahler denominator are 
contained in $\{0\} \cup \ru$. This relies on the deep results on Mahler's method by 
Nishioka, Philippon, Fernandes, as well as by Adamczewski and Faverjon that were recalled in 
\cref{sec:back-mahler}. But first we need the following easy lemma.

\begin{lemma} \label{l:radius-1}
  Let $f(z)=\sum_{n=0}^\infty a_n z^n \in \power{\Qbar}$ be a power series that is \emph{not} a polynomial.
  If $h(a_n) \in o(n)$, then $f(z)$ has radius of convergence $1$ for every absolute value of $\Qbar$.
\end{lemma}

\begin{proof}
  Fix an absolute value $\abs{\cdot}$ on $\Qbar$, and let
  \[
    \rho = \Big( \limsup_{n\to \infty} \sqrt[n]{\abs{a_n}} \Big)^{-1} \in \bR_{\ge 0} \cup \{ \infty\}
  \]
  be the radius of convergence of $f(z)$.

  We show $\rho=1$ by contradiction.
  Suppose first $\rho > 1$.
  Choose $\rho' \in \bR_{>0}$ with with $1 < \rho' < \rho$.
  Since $\limsup_{n\to\infty} \sqrt[n]{\abs{a_n}} = 1 / \rho < 1/\rho'$, we have 
  $\abs{a_n} \le (1/\rho')^{n}$ for all sufficiently large $n$.
  Since $f$ is not a polynomial, there exist infinitely many such $n$ with $a_n \ne 0$, and for these 
  $\abs{a_n^{-1}} \ge (\rho')^n$.
  One has $h(a_n)=h(a_n^{-1})$ as a consequence of the product formula.
  It follows that $h(a_n) = h(a_n^{-1}) \ge \log \abs{a_n^{-1}} \ge n \log(\rho')$, in contradiction to our assumption.

  Suppose now $\rho < 1$, and choose $\rho < \rho' < 1$.
  Then, for all $n_0 \ge 0$, there exists an $n \ge n_0$ such that $\abs{a_n} \ge (1/\rho')^n$.
  Hence $h(a_n) \ge \log \abs{a_n} \ge n \log(1/\rho')$ again yields a contradiction.
\end{proof}

We also require the notion of Cartier operators in the next proof.

\begin{definition} \label{d:cartier}
For every $r \in \digits$, we define a \emph{Cartier operator} 
$\Delta_r \colon \power{\Qbar} \to \power{\Qbar}$ by
\[
  \Delta_r\Big(\sum_{n=0}^\infty a_nz^n\Big) = \sum_{n=0}^\infty a_{kn+r} z^n\,.
\]
\end{definition}
Note that, if $p(z) \in \Qbar[z]$, then $\deg(\Delta_r(p(z))) \le (\deg p)/k$.
Moreover, if $j \ge 1$, then a short computation yields
\[
  \Delta_r\Big(p(z) f(z^{k^j})\Big) = \Delta_r(p(z)) \cdot f(z^{k^{j-1}}).
\]

\begin{proposition} \label{prop: radius}
  Let $f(z) \in \power{\Qbar}$ be $k$-Mahler and let $\mathfrak d(z) \in \Qbar[z]$ be its $k$-Mahler denominator.
  If $\lambda \in \Qbar$ is a root of $\mathfrak d(z)$, and $\abs{\cdot}$ is an absolute value on $\Qbar$ with 
  $0 < \abs{\lambda} < 1$, then the radius of convergence of $f(z)$ with respect to this absolute value 
  is strictly less than $1$.
\end{proposition}

\begin{proof}
  Let us first consider a minimal homogeneous equation associated with $f(z)$:
  \begin{equation}\label{eq: min}
    p_0(z)f(z) + p_1(z)f(z^k)+\cdots+p_d(z)f(z^{k^d})=0 \,.
  \end{equation}
  By minimal, we mean that $p_0(z)p_d(z) \ne 0$, that $d$ is minimal, and that $p_0(z)$, $\ldots\,$,~$p_d(z) \in \Qbar[z]$ are relatively prime.
  If $f(z) = 0$, then we can take $d=0$ and $p_0(z) = 1$.
  Thus also $\mathfrak d(z)=1$ and the claim is trivially true.
  We may assume $f(z) \ne 0$, so that $d \ge 1$.

  As in \cref{sec:back-mahler}, we can assume that there exists a number field $K$ containing $\lambda$ and 
  all coefficients of $f(z)$ as well as the coefficients of the polynomials $p_0(z)$, $\ldots\,$,~$p_d(z)$.
  Further, $\abs{\cdot}$ on $K$ arises from a place $\sv$ of $K$, and $C_\sv$ is the algebraic closure 
  of the completion $K_\sv$, with $\overline{K}$ denoting the algebraic closure of $K$ inside $C_\sv$.

  We claim that  $f(z)$, $\ldots\,$,~$f(z^{k^{d-1}})$ are linearly independent over $\overline{K}(z)$.
  By way of contradiction, suppose that this is not the case.
  Then there exist $q_0(z)$, $\ldots\,$,~$q_{d-1}(z) \in \overline{K}[z]$, not all zero, such that
  \[
    q_0(z) f(z) + \cdots + q_{d-1}(z) f(z^{k^{d-1}}) = 0.
  \]
  Since the coefficients of $q_0(z)$, $\ldots\,$,~$q_d(z)$ are all algebraic over $\bQ$, we can  assume $q_0(z)$, $\ldots\,$,~$q_d(z) \in \Qbar[z]$.
  Let $t \in \{0,\ldots,d-1\}$ be minimal with $q_{t}(z) \ne 0$, say $q_t(z) = \sum_{j=m}^M b_j z^j$ with $b_m$,~$b_M \ne 0$.
  Let $m = \sum_{\nu=0}^\infty k^\nu m_\nu$ with $m_\nu \in \digits$ be the base-$k$ expansion of $m$ (with all but finitely many $m_\nu$ being zero).
  Setting $\Delta = \Delta_{m_t}\circ \cdots \circ \Delta_{m_0}$, we see $\Delta(q_t(z)) \ne 0$.
  Then
  \[
    \Delta(q_t(z)) f(z) + \cdots + \Delta(q_{d-1}(z)) f(z^{k^{d-1-t}}) = 0
  \]
  is a $k$-Mahler equation for $f(z)$, contradicting the minimality of $d$.
  Therefore $f(z)$, $\ldots\,$,~$f(z^{k^{d-1}})$ must be linearly independent over $\overline{K}(z)$, as claimed.

  Now
  \[
    \begin{pmatrix}
      f(z) \\
      \vdots \\
      f(z^{k^{d-1}})
    \end{pmatrix}
    = A(z)
    \begin{pmatrix}
      f(z^k) \\
      \vdots \\
      f(z^{k^d})
    \end{pmatrix}
  \]
  with
  \[
    A(z) =
    \begin{pmatrix}
      -\frac{p_1(z)}{p_0(z)} & -\tfrac{p_2(z)}{p_0(z)} & \cdots & -\tfrac{p_{d-1}(z)}{p_0(z)}  & -\tfrac{p_{d}(z)}{p_0(z)} \\
      1 & 0 & \cdots & 0 & 0 \\
      0 & 1 & \ddots & \vdots & \vdots \\
      \vdots & \ddots & \ddots & 0 & \vdots \\
      0 & \cdots & 0 & 1 & 0 \\
    \end{pmatrix} \in \GL_d(K(z)).
  \]
  By \cref{cor: lifting}, we have that
  \[
    f(\lambda^{k^n}), f(\lambda^{k^{n+1}}), \ldots, f(\lambda^{k^{n+d-1}})
  \]
  are linearly independent over $\overline{K}$, as soon as $n$ is large enough, say $n\geq n_0$.
  
  Now, iterating Equation \eqref{eq: min}, we obtain an equation of the form
  \begin{equation}\label{eq: iterate}
    r_0(z)f(z) + r_1(z) f(z^{k^{n_0}})+\cdots+r_d(z)f(z^{k^{n_0+d-1}})=0 \,,
  \end{equation}
  where we assume without any loss of generality that $r_0(z)$, $\ldots\,$,~$r_d(z) \in K[z]$ are relatively prime. 
  We claim that $f(z)$ has a pole at $\lambda$. 
  Let us assume by contradiction that $f(z)$ is well-defined at $\lambda$.
  Since $\mathfrak d(\lambda)=0$, it follows that $r_0(\lambda)=0$ and we get that 
  \[
  r_1(\lambda)f(\lambda^{k^{n_0}})+\cdots+r_d(\lambda)f(\lambda^{k^{n_0+d-1}})=0 \,.
  \]
  Since $f(\lambda^{k^{n_0}})$, $\ldots\,$,~$f(\lambda^{k^{n_0+d-1}})$ are linearly independent over 
  $\overline{K}$, all the $r_i(z)$ should vanish at $\lambda$, contradicting the fact that they are 
  relatively prime. Hence, $f(z)$ has a pole at $\lambda$ and its radius of convergence is therefore 
  less than $1$.  
\end{proof}

We now have the ingredients to characterize Mahler functions with $h(a_n) \in o(n)$.

\begin{proof}[Proof of \cref{t-main:hn}]
  Let $f(z) = \sum_{n=0}^\infty a_n z^n \in \power{\Qbar}$ be $k$-Mahler.
  
  \ref{hn:hon}$\,\Rightarrow\,$\ref{hn:radius}
  Suppose $h(a_n) \in o(n)$.
  By \cref{l:radius-1} the series $f(z)$ has radius of convergence at least $1$ with respect to every absolute 
  value $\abs{\cdot}$ on $\Qbar$.

  \ref{hn:radius}$\,\Rightarrow\,$\ref{hn:denominator}
  Suppose now $f(z)$ has radius of convergence at least $1$ with respect to every absolute value $\abs{\cdot}$ 
  on $\Qbar$. Let $\mathfrak d(z) \in \Qbar[z]$ be the $k$-Mahler denominator of $f(z)$.
  Suppose there exists $\lambda \in \Qbar \setminus \{0\}$ with $\mathfrak d(\lambda)=0$ such that 
  $\lambda$ is \emph{not} a root of unity.
  By Kronecker's Theorem there exists an absolute value $\abs{\cdot}$ on $\Qbar$ for which $\abs{\lambda} < 1$.
  By \cref{prop: radius}, the series $f(z)$ has radius of convergence strictly less than $1$ for this absolute value, a contradiction.

  \ref{hn:denominator}$\, \Rightarrow\,$\ref{hn:hOlog2}
  Suppose all roots of the $k$-Mahler denominator $\mathfrak d(z) \in \Qbar[z]$ of $f(z)$ are contained in 
  $\{0\} \cup \ru$. Then $h(a_n) \in O(\log^2n)$ by \ref{upper-generic:logn2} of \cref{p:upper-generic}.
  
  \ref{hn:hOlog2}$\, \Rightarrow\, $\ref{hn:hon} Clearly $h(a_n) \in O(\log^2n)$ implies $h(a_n) \in o(n)$.
\end{proof}

\section{Second gap: characterization of regular Mahler functions}\label{sec:hlog2}

In this section, we characterize Mahler functions 
$f(z) = \sum_{n=0}^\infty a_n z^n \in \power{\Qbar}$ with $h(a_n) \in o(\log^2 n)$. 
The following result also proves Case~\ref{aut-reg:reg} of Theorem \ref{thm: aut-reg}.

\begin{theorem} \label{t-main:hlog2}
  Let $f(z) = \sum_{n=0}^\infty a_n z^n \in \power{\Qbar}$ be a $k$-Mahler function.
  The following statements are equivalent.
  \begin{equivenumerate}
  \item \label{hlog2:holog2} We have $h(a_n) \in o(\log^2n)$.
  \item \label{hlog2:roots} Every non-zero root of the $k$-Mahler denominator of $f(z)$ 
  belongs to $\ru_k$. 
  \item \label{hlog2:regular}  The power series $f(z)$ is  $k$-regular. 
  \item \label{hlog2:hOlog} We have $h(a_n) \in O(\log n)$.
  \end{equivenumerate}
\end{theorem}

We already know that, if $h(a_n) \in o(\log^2n)$, then every root $\zeta$ of the $k$-Mahler 
denominator $\mathfrak d(z)$ of $f(z)$ is contained in $\{0\} \cup \ru$. 
The brunt of the work in this section lies in showing $\zeta \in \{0\} \cup \rucommon$, that is, 
if $\zeta \ne 0$, then $\zeta^{k^j} \ne \zeta$ for all $j > 0$. This requires a careful analysis of the 
asymptotics of $f(z)$ at such a hypothetical root of $\mathfrak d(z)$ to establish a contradiction.

We start with some estimates.

\begin{lemma}
  Let $f(z) = \sum_{n = 0}^\infty a_n z^n \in \power{\bR}$ be a power series with non-negative coefficients.
  Suppose there exists $c \in \bR_{>0}$ with $a_n \le n^{c \log n}$ for all sufficiently large $n$.
  Let $c'$,~$\varepsilon \in \bR_{>0}$ with $c' > 2c$.
  Then there exists $t_0 \in [0,1)$ such that
  \[
    \sum_{n=\lceil m \log^2m \rceil}^\infty a_n t^n < 
    \varepsilon \qquad\text{for all $t \in [t_0,1)$ and $m \ge \frac{c'}{1-t}$}\,\cdot
  \]
\end{lemma}

\begin{proof}
  By assumption, for large $n$,
  \[
    a_n t^n  \le \exp(c \log^2 n + n \log t).
  \]
  We will show that, for sufficiently large $m$ (ensured by choice of $t_0$) and  $n \ge m \log^2 m$,
  \begin{equation} \label{eq:1}
   c \log^2 n  + n \log t \le \tfrac{1}{2} n \log t.
 \end{equation}
  We first show how to conclude the proof using \cref{eq:1}.
  Then $a_n t^n \le t^{n/2}$ and
  \[
    \sum_{n=\lceil m \log^2 m \rceil}^\infty a_n t^n \le \frac{t^{\lceil m \log^2 m \rceil/2}}{1-\sqrt{t}} 
    = \frac{(1 + \sqrt{t}) t^{\lceil m \log^2 m \rceil/2}}{1-t} < \frac{2t^{\lceil m \log^2 m \rceil/2}}{1-t}\, \cdot
  \]
  We need to bound the right side by a constant.
  Using $m \ge c'/(1-t)$ and $t \in [0,1)$, we have
  \begin{equation} \label{eq:limit}
    \log\left( \frac{2t^{\lceil m \log^2 m \rceil/2}}{1-t} \right) \le \log 2 + \frac{c'\log t}{2(1-t)} 
    \log^2\left(\frac{c'}{1-t}\right) - \log(1-t).
  \end{equation}
  Recall $\lim_{t\to 1} \log t /(1-t) = -1$ and $\log^2(c'/(1-t)) \sim \log^2(1-t)$ for $t\to 1^{-}$.
  Hence the right side of \cref{eq:limit} tends to $-\infty$ as $t \to 1^{-}$.
  Choosing $t_0 \in [0,1)$ sufficiently close to $1$, therefore
  \[
    \sum_{n=\lceil m \log^2 m \rceil}^\infty a_n t^n \le \varepsilon 
    \qquad\text{for $t \in [t_0,1)$ and $m \ge c'/(1-t)$.}
  \]

  It remains to show the bound in \cref{eq:1}. The latter 
  is equivalent to $c \log^2 n + \tfrac{1}{2} n \log t \le 0$.
  Since $\log t \le t - 1 \le -c'/m$, it suffices to show
  \begin{equation} \label{l:ineq}
    c \log^2 n - n \frac{c'}{2m} \le 0 \qquad\text{ for $n \ge m \log^2 m$}.
  \end{equation}
  We first show this for $n = m \log^2 m$.
  Now
  \[
    \begin{split}
      c \log^2(m \log^2 m) - m (\log^2 m) \frac{c'}{2m} \sim (c - c'/2) \log^2 m
    \end{split}
  \]
  as a function in $m$ for $m \to \infty$, and $c - c'/2$ is negative by choice of $c'$.
  Thus, for sufficiently large $m$,  we have $c \log^2(m \log^2 m) - m (\log^2 m) \frac{c'}{2m} \le 0$.
  We can ensure a large enough $m$ by choosing $t_0 \in [0,1)$ sufficiently close to $1$.

  Now, set $g(n) := c \log^2n$ and $h(n) := n \frac{c'}{2m}$.
  Then $g'(n) = \frac{2 c \log n}{n}$, and hence
  \[
    g'(m \log^2m) = \frac{2c \log m + 2c \log(\log^2 m)}{m \log^2m} \sim \frac{2c}{m \log m}\,\cdot
  \]
  Thus, choosing $m$ sufficiently large, we may also ensure
  \[
    g'(m \log^2m) \le h'(m \log^2m) = \frac{c'}{2m}\,\cdot
  \]
  Since $g(n)$ is concave for $n \ge \exp(1)$, this ensures $g(n) \le h(n)$ for $n \ge m \log^2m$. 
  This proves \cref{eq:1} and ends the proof of the lemma. 
\end{proof}

\begin{lemma} \label{l:estimate}
  Let $f(z) = \sum_{n=0}^\infty a_n z^n \in \power{\bR}$ be a power series with non-negative coefficients.
  Let $a \in \bR_{\ge 0}$ and $b$,~$c \in \bR_{>0}$.
  Assume that there exist a sequence $(t_j)_{j \ge 0} \to 1$ in $[0,1)$ and $m_0 \in \bZ_{\ge 0}$ such that
  \[
    f(t_j) \ge (1-t_j)^a \exp(c \log^2m) t_j^{m b} \qquad\text{for all $j \ge 0$ and $m \ge m_0$.}
  \]
  Then there exist $c' \in \bR_{>0}$ and infinitely many $n \ge 1$ with $a_n > \exp(c' \log^2n)$.
\end{lemma}

\begin{proof}
  Without any loss of generality we can assume that there exists a constant $c'>2c$ 
  such that $a_n \le \exp(c' \log(n)^2)$ 
  for all sufficiently large $n$. Indeed, otherwise the result holds trivially.  
  By the previous lemma, for all sufficiently large $j$ and $m \ge 3c'/(1-t_j)$, we have 
  \[
    \sum_{n=\lceil m \log^2m \rceil}^\infty a_n t_j^n \le 1 \,.
  \]

  Let $m_j = \lceil 3c'/(1-t_j) \rceil$ and
  \[
    \begin{split}
      A_j &= \log\big((1-t_j)^a \exp(c \log^2 m_j) t_j^{m_j b}\big)\\
      &= a \log(1 - t_j) + c \log^2m_j  + m_j b \log t_j \,.
    \end{split}
  \]
  Then, for sufficiently large $j$,
  \[
    \sum_{n=0}^{\lfloor m_j \log^2m_j \rfloor} a_n t_j^n \ge \exp(A_j) - 1 \,.
  \]

  Since
  \[
     \Big( \frac{3c'}{(1-t_j)} + 1\Big) b \log t_j  \le m_j b \log t_j \le \frac{3c'}{(1-t_j)} b \log t_j \,,
  \]
  and $((\log t_j) /(1-t_j))_{j \ge 0} \to -1$, we find $(m_j b \log t_j)_{j \ge 0} \to -3c'b$. 
Since $\log^2(1/(1-t_j)) = \log^2(1-t_j)$ and therefore $\log^2 m_j \sim \log^2(1-t_j)$, we see that 
$A_j \sim c \log^2m_j$. 
  Choosing $j$ sufficiently large, we may assume
  \[
    A_j \ge \frac{c}{2} \log^2m_j \,.
  \]
  Therefore, again restricting to large enough $j$ for the last inequality,
  \[
    \sum_{n=0}^{\lfloor m_j \log^2m_j \rfloor} a_n t_j^n \ge \exp(A_j) - 1  
    \ge \exp\Big(\frac{c}{4} \log^2m_j \Big) \,.
  \]
  
  By the pigeonhole principle, there exists $0 \le n_j \le m_j \log^2m_j$ such that
  \[
    a_{n_j} \ge \exp\Big(\frac{c}{4} \log^2m_j\Big) / (1+ m_j \log^2m_j )\,.
  \]
  Thus
  \[
    \log a_{n_j} \ge \frac{c}{4} \log^2m_j - \log(m_j \log^2 m_j + 1) \sim \frac{c}{4} \log^2m_j\,.
  \]
  We may assume $\log a_{n_j} \ge \frac{c}{8} \log^2m_j$.
  To finish, since $n_j \le m_j \log^2m_j$, we have
  \[
    \begin{split}
      \log^2n_j  &\le \log^2(m_j \log^2 m_j) \sim \log^2m_j \,.
    \end{split}
  \]
  We may take $\log^2n_j \le 2\log^2m_j$, so that $\log a_{n_j} \ge \frac{c}{16} \log^2n_j$. 
  Since $\log a_{n_j} \ge \frac{c}{8} \log^2m_j$ and $(m_j)_{j \ge 0} \to \infty$, also $(n_j)_{j \ge 0} \to \infty$.
  Thus there are in fact infinitely many distinct such $n_j$.
\end{proof}

\begin{lemma} \label{l:bound-other-roots}
  Let $\zeta \in \bC$ with $\zeta^k = \zeta$.
  Let $p(z) \in \bC[z]$ with $p(0)=1$ and $p(\zeta) \ne 0$.
  Then there exists $c \in \bR_{>0}$ such that, for all $t \in [0,1)$ with $p(\zeta t^{k^n}) \ne 0$ 
  for all $n \ge 0$,
  \[
    \abs[\bigg]{ \displaystyle\left(\prod_{n=0}^\infty p(\zeta t^{k^n})\right)^{-1} } > \abs{1-t}^c\, .
  \]
\end{lemma}

\begin{proof}
  The proof is the same as the one of the lower bound in \cite[Lemma 9.5 and Proposition 9.2]{adamczewski-bell17}.
  Let $\alpha_1$, $\ldots\,$,~$\alpha_s$ denote the roots of $p(z)$ (with multiplicity).
  Then
  \[
    p(z) = (1-\alpha_1^{-1} z)\cdots (1-\alpha_s^{-1} z)\,.
  \]
  It suffices to show the claim for $1-\alpha_1^{-1}z$.
  Suppose $t \in [0,1]$ is such that $\zeta t^{k^n} \ne \alpha_1$ for all $n \ge 0$.
  Then the infinite product $\prod_{n=0}^\infty (1-\alpha_1^{-1}t^{k^n})^{-1}$ converges, and
  \[
    \abs[\bigg]{\prod_{n=0}^\infty \frac{1}{1-\alpha_1^{-1} \zeta t^{k^n}}}  
    \ge \prod_{n=0}^\infty \frac{1}{1+\abs{\alpha_1^{-1}}t^{k^n}} 
    \ge \prod_{n=0}^\infty \exp(-\abs{\alpha_1^{-1}}t^{k^n}) \,.
  \]
  Then, by \cite[Lemma 9.4]{adamczewski-bell17},
  \[
    \prod_{n=0}^\infty \exp(-\abs{\alpha_1^{-1}}t^{k^n}) \ge 
    \exp\Big( - \abs{\alpha_1^{-1}} (1-1/k)^{-1} \sum_{n=1}^\infty \frac{t^n}{n} \Big) 
    = (1-t)^{\frac{\abs{\alpha_1^{-1}}k}{k-1}}\,. \qedhere
  \]
\end{proof}

Let $B(\lambda,r) \subseteq \bC$, respectively $\overline{B(\lambda,r)} \subseteq \bC$, 
denote the open, respectively closed, disc of radius $r \in \bR_{\ge 0}$ with center $\lambda \in \bC$.

\begin{lemma}[{Special case of \cite[Lemma 10.2]{adamczewski-bell17}}] \label{l:bound-matrix-system}
  Let $d \in \bZ_{>0}$, let $\zeta \in\bC\setminus \{0\}$ such that $\zeta^k = \zeta$, and 
  let $A \colon \overline{B(0,1)} \to \bC^{d \times d}$ be a continuous, matrix-valued function.
  Assume that $w(z) \in \power{\bC}^d$ satisfies the equation
  \[
    w(\lambda) = A(\lambda) w(\lambda^k) \qquad\text{for all $\lambda \in B(0,1)$\,.}
  \]
  Assume also that the following properties hold.
  \begin{propenumerate}
  \item\label{bms:analytic} The coordinates of $w(z)$ are analytic in $B(0,1)$.
  \item\label{bms:nilpotent} The matrix $A(\zeta)$ is \emph{not} nilpotent.
  \item\label{bms:independent} The set $\{\, w(\lambda) : \lambda \in B(0,1) \,\}$ is not contained 
  in a proper vector subspace of $\bC^d$.
  \end{propenumerate}
  Then there exist $c \in \bR_{>0}$ and a sequence $(t_j)_{j \ge 0} \to 1$ in $[0,1)$ such that
  \[
    \norm{w(t_j\zeta)} > \abs{1-t_j}^c \qquad\text{for all $j \ge 0$}\,.
  \]
\end{lemma}

\begin{proof}
  This is \cite[Lemma 10.2]{adamczewski-bell17} in the special case $\theta=0$.
  We do not assume $w(z)$ to be continuous in $\overline{B(0,1)}$, but this assumption 
  is never used in the proof and is therefore superfluous.
\end{proof}

\begin{lemma} \label{l:remove-root}
  Let $b \in \bZ_{>0}$ and ~$a$,~$a' \in \bR$ with $a' > a>0$.
  Then there exists $t_0 \in [0,1)$ such that
  \[
    (1 - t^{1/b})^a > (1-t)^{a'} \qquad\text{for all $t \in [t_0,1)$.}
  \]
\end{lemma}

\begin{proof}
  For $t \in [0,1)$ we have
  \[
    1-t = (1-t^{1/b}) \sum_{i=0}^{b-1} t^{i/b} \le b(1-t^{1/b}).
  \]
  Moreover $(1-t)^{a'-a} < 1/b^a$ for $t$ sufficiently close to $1$.
  Then
  \[
    (1-t^{1/b})^a \ge \frac{(1-t)^{a'}}{(1-t)^{a'-a}b^a} > (1-t)^{a'}. \qedhere
  \]
\end{proof}

Armed with these estimates, we can finally prove a further restriction on the roots of the 
$k$-Mahler denominator. This is the key step in the current section.
The arguments are in many aspects very similar to those used by Adamczewski and Bell 
in \cite[\S11]{adamczewski-bell17}.

\begin{proposition} \label{p:coprime-k-lower}
  Let $f(z)=\sum_{n=0}^\infty a_n z^n \in \power{\bC}$ be a $k$-Mahler series that is analytic in $B(0,1)$.
  Let $\zeta \in \cU$ with $\zeta^{k^{j_0}} = \zeta$ for some $j_0 \ge 1$, and let $l = k^{j_0}$.
  Suppose there exists an $l$-Mahler equation
  \[
    p_0(z) f(z) = p_1(z) f(z^{l}) + \cdots + p_d(z) f(z^{l^d}),
  \]
  with $p_0(z),\ldots,p_d(z) \in \bC[z]$ coprime, with $p_0(z)p_d(z) \ne 0$, and such that $p_0(\zeta) = 0$.

  Then there exist $a$,~$b$,~$c \in \bR_{>0}$, $m_0$,~$n_0 \in \bZ_{\ge 0}$, 
  and a sequence $(t_j)_{j \ge 0} \to 1$ in $[0,1)$ such that
  \[
    \left\vert\sum_{n=0}^\infty a_{n+n_0} (\zeta t_j)^n \right\vert \ge 
    (1-t_j)^a \exp(b \log^2m) t_j^{mc} \qquad\text{for all $j \ge 0$ and $m \ge m_0$\,.}
  \]
\end{proposition}

\begin{proof}
  Applying \cref{l:kill-0}, there exist $n_0 \ge 0$ and $q_0(z),\ldots,q_{d+1}(z) \in \bC[z]$ such that $a_{n_0} \ne 0$ and
  \[
    f_0(z) = \sum_{n=0}^\infty a_{n+n_0}z^n
  \]
  satisfies
  \[
    q_0(z) f_0(z) = q_1(z) f_0(z^{l}) + \cdots + q_{d+1} f_0(z^{l^{d+1}})\,,
  \]
  with $q_0(0)=1$, with $q_0(\zeta)=0$, and with $q_i(\zeta) \ne 0$ for some $i \in \{1,\ldots,d+1\}$.

  Let $\nu_i \in \bZ_{\ge 0}$ be the order of vanishing of $q_i(z)$ at $\zeta$.
  Define
  \[
    r \coloneqq \min\Big\{\, \frac{\nu_i + (i-1)\nu_0}{i} : i \in 1, \ldots, d+1\,\Big\} \in \bQ_{\ge 0} \,.
  \]
  Since $\nu_0 > 0$ and $\nu_i=0$ for some $i \in \{1, \ldots, d+1\}$, we have $r < \nu_0$.
  Defining
  \[
    g(z) \coloneqq f_0(\zeta z) \prod_{n=0}^\infty \frac{q_0(\zeta z^{l^n}) }{(1 - z^{l^n})^r}
  \]
  we obtain
  \[
    g(z) = \sum_{i=1}^{d+1} r_i(z) g(z^{l^i}) \qquad\text{with}\qquad r_i(z) = q_i(\zeta z) \frac{1}{(1-z)^r} 
    \prod_{n=1}^{i-1} \frac{q_0(\zeta z^{l^n})}{(1 - z^{l^n})^r} \in \bC(z)\,.
  \]
  In the expression for $r_i(z)$, the denominator has roots at every $\omega \in \bC$ for which 
  $\omega^{l^{i-1}}=1$. If $\omega \ne 1$, then $r < \nu_0$ guarantees that $r_i(z)$ does not actually 
  have a pole at $\omega$.
  For $\omega=1$, this is ensured by $ir \le \nu_i + (i-1)\nu_0$.
  Thus, all $r_i(z)$ are in fact polynomials.
  Moreover, by choice of $r$, there exists an $i_0 \in \{1,\ldots,d+1\}$ such that $r_{i_0}(1) \ne 0$.

\medskip

 \noindent \textbf{Claim:} There exist $a \in \bR_{>0}$ and a sequence $(t_j)_{j \ge 0} \to 1$ in $[0,1)$ with
  \[
    \abs{g(t_j)} \ge (1-t_j)^a.
  \]

 \noindent \emph{Proof of Claim.}
  First we deal with the degenerate case in which $g(z)$ is constant.
  Then $g(z) = g(0)$ and from the definition of $g$ we see $g(0) \ne 0$ since $f_0(0) \ne 0$ and $q_0(0)=1$.
  Choosing $a=1$, any sequence $(t_j)_{j \ge 0} \to 1$ in $[0,1)$ satisfies $\abs{g(t_j)} > 1-t_j$ for sufficiently large $j$.
  From now on we may assume that $g(z)$ is \emph{not} constant.
  
  We are going to apply \cref{l:bound-matrix-system}.
  Denote by $\norm{\cdot}$ the maximum norm with respect to $\abs{\cdot}$.
  Let $w(z) = \begin{pmatrix} g(z), g(z^l), \dots, g(z^{l^{d}}) \end{pmatrix}^T$ and
  \[
    A(z) =
    \begin{pmatrix}
      r_1(z) & r_2(z) & \dots & r_{d-1}(z) & r_d(z) & r_{d+1}(z) \\
      1 & 0 & \dots & 0 & 0 & 0 \\
      0 & 1 & \dots & 0 & 0 & 0 \\
      \vdots & 0 & \ddots & \vdots & \vdots & \vdots \\
      \vdots & \vdots & \dots & 1 & 0 & 0 \\
      0 & 0 & \dots & 0 & 1 & 0 \\
    \end{pmatrix}
    \in \bC(z)^{(d+1) \times (d+1)}.
  \]
  Then $w(z) = A(z)w(z^l)$.
  The coordinates of $A(z)$ are polynomials and hence of course continuous.
  We verify the conditions of \cref{l:bound-matrix-system}.

\noindent \ref{bms:analytic} The coordinates of $w(z)$ are analytic in $B(0,1)$ since $g(z)$ is analytic in $B(0,1)$.

\noindent   \ref{bms:nilpotent} The characteristic polynomial of $A(z)$ is $y^{d+1} - r_1(z) y^d - \cdots  - r_{d+1}(z) 
  \in \bC(z)[y]$.
  Since $r_{i_0}(1) \ne 0$, the matrix $A(1)$ is not nilpotent. 

\noindent   \ref{bms:independent}
  Suppose that $S = \{\, w(\lambda) : \lambda \in B(0,1) \,\}$ is contained in a proper subspace of $\bC^d$.
  Then there exist $\alpha_0$, $\ldots\,$,~$\alpha_d \in \bC$, not all zero, such that 
  $\alpha_0 g(\lambda) + \cdots + \alpha_d g(\lambda^{l^d}) = 0$ for all $\lambda \in B(0,1)$.
  Since $g(z)$ is analytic in $B(0,1)$ this forces $\alpha_0 g(z) + \cdots + \alpha_d g(z^{l^d}) = 0$. But then $g(z)$ is constant by \cite[Lemma 7.9]{adamczewski-bell17}, a contradiction.
  Hence the set $S$ is not contained in a proper subspace of $\bC^d$.

  Applying \cref{l:bound-matrix-system}, there exist $a \in \bR_{>0}$ and a sequence $(t_j)_{j\ge 0} \to 1$ in $[0,1)$ such that
  \[
    \norm{w(t_j)} > (1-t_j)^a \qquad \text{for all $j \ge 0$}\,.
  \]
  Restricting to a subsequence and making a substitution, we may assume that there exists $i \in [0,d]$ 
  and $b = l^{i}$ such that $\abs{g(t_j)} > (1-t_j^{1/b})^a$ for $j \ge 0$.
  Applying \cref{l:remove-root} and replacing $a$ by a slightly larger constant, we may actually take $\abs{g(t_j)} > (1-t_j)^a$ for all $j \ge 0$. \hfill$\square$ (Claim)

  By the result of Mahler, see Example~\ref{ex:cyclo} in Section \ref{t:k-partition}, there exists a constant $c \in \bR_{>0}$ such that, for some $m_0 \ge 0$,
  \[
    \prod_{n=0}^\infty (1-t_j^{l^n})^{-1} \ge \sum_{n=m_0}^\infty \exp(c \log^2n) t_j^n\,.
  \]
  Thus
  \[
    \prod_{n=0}^\infty (1-t_j^{l^n})^{-1} \ge \exp(c \log^2m) t_j^m \qquad \text{for all $m \ge m_0$\,.}
  \]

  The lower bound for $g(z)$ together with the fact that $f_0(z)$ is analytic in $B(0,1)$ implies 
  $q_0(\zeta t_j^{l^n}) \ne 0$ for all $n \ge 0$.
  By \cref{l:bound-other-roots}, there exists $a' \in \bR_{>0}$ such that
  \[
    \abs[\bigg]{ \prod_{n=0}^\infty \frac{(1-t_j^{l^n})^{\nu_0}}{q_0(\zeta t_j^{l^n})} } > (1 - t_j)^{a'} 
    \qquad\text{for $j$ large enough.}
  \]
  With $b = \nu_0 - r > 0$, we conclude, for $m \ge m_0$, that 
  \[
    \abs{f_0(t_j\zeta)} = \abs{g(t_j)} \cdot \abs[\bigg]{\prod_{n=0}^\infty \frac{(1-t_j^{l^n})^r}{q_0(\zeta t_j^{l^n})}} 
    > (1-t_j)^{a+a'} \exp(c b \log^2m ) t_j^{m b}\, . \qedhere
  \]
\end{proof}

\begin{proposition} \label{p:goodroots}
  Let $f(z)=\sum_{n=0}^\infty a_n z^n \in \power{\Qbar}$ be $k$-Mahler and suppose that 
  $h(a_n) \in o(\log^2n)$.
  Then the roots of the $k$-Mahler denominator of $f(z)$ are contained in $\{0\} \cup \rucommon$.
\end{proposition}

\begin{proof}
  Let $\mathfrak d(z)$ be the $k$-Mahler denominator of $f(z)$.
  Suppose that $\zeta \in \Qbar \setminus \{0\}$ is such that $\mathfrak d(\zeta)=0$.
  Then $\zeta \in \ru$ by \cref{t-main:hn}.
  We have to show $\zeta^{k^j} \ne \zeta$ for all $j \ge 1$.

  Suppose to the contrary that $\zeta^{k^{j_0}}= \zeta$ for some $j_0 \ge 1$.
  Let $l = k^{j_0}$.
  Then $f(z)$ is also $l$-Mahler by \cref{l:mahler-power}.
  Let $p_0(z),\ldots,p_d(z) \in \Qbar[z]$ be coprime, with $p_0(z)p_d(z) \ne 0$, such that
  \[
    p_0(z) f(z) = p_1(z) f(z^{l}) + \cdots + p_d(z) f(z^{l^d})\,.
  \]
  Since this is also a $k$-Mahler equation for $f(z)$, the $k$-Mahler denominator $\mathfrak d(z)$ 
  divides $p_0(z)$, hence $p_0(\zeta) = 0$.
  Fix any embedding $\Qbar \hookrightarrow \bC$, and thereby an archimedean absolute value on $\Qbar$.
  We apply \cref{p:coprime-k-lower} to conclude that there exist 
  $a,b,c \in \bR_{>0}$, $m_0,n_0 \in \bZ_{\ge 0}$, and a sequence $(t_j)_{j \ge 0} \to 1$ in $[0,1)$ 
  such that
  \[
    \abs{f_0(\zeta t_j)} \ge (1-t_j)^a \exp(b \log^2m) t_j^{mc} \qquad\text{for all $j \ge 0$ and $m \ge m_0$\,,}
  \]
  where $f_0(z) = \sum_{n=0}^\infty a_{n+n_0} z^n$. 
  Since $\sum_{n=0}^\infty \abs{a_{n+n_0}} t^n \ge 
  \abs{f_0(\zeta t)}$ for $t \in [0,1)$, the conditions of \cref{l:estimate} are satisfied for 
  $\sum_{n=0}^\infty \abs{a_{n+n_0}} z^n$.
  Thus there exist $c' \in \bR_{>0}$ such that $\abs{a_n} \ge \exp(c' \log^2n)$ infinitely often.
  Thus $h(a_n) \not\in o(\log^2n)$; a contradiction.
\end{proof}

Once we know that all roots of the $k$-Mahler denominator of $f(z)$ are contained in 
$\{0\} \cup \rucommon$ it is not hard to show that $f(z)$ is $k$-regular.
This was shown by Dumas \cite[Th\'eor\`eme 30]{dumas93}; 
see also \cite[Proposition 2]{bell-chyzak-coons-dumas18}.
We recall the proof.

Keep in mind that $\rucommon$ consists of all roots of unity $\zeta$ for which $\zeta^{k^j} \ne \zeta$ for all $j \ge 1$.
In particular, $1 \not \in \rucommon$.
By \cite[Proposition 7.8]{adamczewski-bell17}, the infinite product
\[
  \prod_{n=0}^\infty (1 - \zeta^{-1} z^{k^n})^{-1}
\]
is $k$-regular for $\zeta \in \rucommon$.

\begin{proposition} \label{p:hlog2}
  Let $f(z)=\sum_{n=0}^\infty a_n z^n \in \power{\Qbar}$ be a $k$-Mahler series 
  and suppose that $h(a_n) \in o(\log^2n)$.
  Then $f(z)$ is a $k$-regular power series.
\end{proposition}

\begin{proof}
  Since every polynomial is $k$-regular, and sums of $k$-regular sequences are $k$-regular, 
  it suffices to show the claim for $\sum_{n=0}^\infty a_{n+n_0}z^n$ for some $n_0 \ge 0$.
  By \cref{l:kill-0} we may therefore assume $\mathfrak d(0)=1$ for the $k$-Mahler denominator $\mathfrak d(z)$ of 
  $f(z)$.
  By \cref{p:goodroots}, all roots of $\mathfrak d(z)$ are contained in $\rucommon$. 
  By \cref{t:dumas-structure} we can write
  \begin{equation*}
    f(z) = \frac{g(z)}{\prod_{n=0}^\infty \mathfrak d(z^{k^n})}
  \end{equation*}
  with a $k$-Becker series $g(z)$.
  By \cref{t:becker}, the series $g(z)$ is $k$-regular.
  Because 
  \[
    \prod_{n=0}^\infty (1 - \zeta^{-1} z^{k^n})^{-1}
  \]
  is $k$-regular for $\zeta \in \rucommon$, and products of $k$-regular series are $k$-regular, also $f(z)$ is $k$-regular.
\end{proof}

Allouche and Shallit \cite[Theorem 2.10]{allouche-shallit92} show $\abs{a_n} \in O(n^c)$ for a 
$\bC$-valued $k$-regular sequence.
A similar argument bounds the height of the coefficients.

\begin{lemma} \label{l:growth-regular}
  If $f(z) = \sum_{n=0}^\infty a_n z^n \in \power{\Qbar}$ is $k$-regular, then $h(a_n) \in O(\log n)$.
\end{lemma}

\begin{proof}
  For $n \in \bZ_{\ge 0}$, we recall that $\langle n\rangle_k \in \digits^*$ is the canonical base-$k$ expansion of $n$.  
  By \cref{t:regular-linrep} there exists a linear representation $(u,\mu,v)$ 
  (of some dimension $d \in \bZ_{\ge 0}$) such that  $a_n = u \mu(\langle n\rangle_k) v$ 
  for all $n \in \bZ_{\ge 0}$.
  Moreover, using basic properties of the logarithmic Weil height (see \cite[Chapter 3]{Miw}), we deduce that 
  $h(u\mu(w)v) \in O(\length{w})$ for $w \in \digits^*$. 
  Noting that $\length{ \langle n\rangle_k } \in O(\log n)$, we obtain $h(a_n) \in O(\log n)$.
\end{proof}

At this point, we are ready to prove \cref{t-main:hlog2}.

\begin{proof}[Proof of \cref{t-main:hlog2}]
  Let $f(z)=\sum_{n=0}^\infty a_n z^n \in \power{\Qbar}$ be a $k$-Mahler function.
   
  \ref{hlog2:holog2}$\,\Rightarrow\,$\ref{hlog2:roots}
  Suppose $h(a_n) \in o(\log^2n)$.
  By \cref{p:goodroots} all roots of the $k$-Mahler denominator of $f(z)$ are contained 
  in $\{0\} \cup \rucommon$.

  \ref{hlog2:roots}$\,\Rightarrow\,$\ref{hlog2:regular}
  Suppose that all roots of the $k$-Mahler denominator are contained in $\{0\} \cup \rucommon$.
  Then $f(z)$ is $k$-regular by \cref{p:hlog2}.

  \ref{hlog2:regular}$\,\Rightarrow\,$\ref{hlog2:hOlog}
  Suppose $f(z)$ is $k$-regular.
  Then $h(a_n) \in O(\log n)$ by \cref{l:growth-regular}.

  \ref{hlog2:hOlog}$\,\Rightarrow\,$\ref{hlog2:holog2}
  Clearly $h(a_n) \in O(\log n)$ implies $h(a_n) \in o(\log^2n)$.
\end{proof}

\section{Third gap:  word-convolution products of automatic sequences}
\label{sec:hlog}

In this section, we characterize Mahler functions 
$f(z)= \sum_{n=0}^\infty a_n z^n\in \power{\Qbar}$ 
with $h(a_n) \in o(\log n)$. 
The arguments are similar to the ones used in \cite{bell05} and \cite{bell-coons-hare16}.
As before, we actually prove a more extensive characterization involving a structural property.

Before  stating the main result of this section, 
we first recall the definition of the word-convolution product following \cite{bell-coons-hare16}. 

\begin{definition}\label{def: wordproduct}
Given two sequences of complex numbers 
$(a(n))_{n\geq 0}$ and $(b(n))_{n\geq 0}$, their \emph{word-convolution product} 
is the sequence $a\star_w b$ defined by  
$$
a\star_w b(n) = \sum_{j=0}^s a([ i_{1}\cdots i_{j} ]_k) b([ i_{j+1}\cdots i_s]_k) \,, 
$$
where $\langle n\rangle_k=i_1i_2\cdots i_s\in \digits^*$.
\end{definition}

We also need the notion of tame semi-group of matrices.

\begin{definition}
  Let $d$ be a positive integer. 
  A semigroup of matrices $\cS \subseteq K^{d \times d}$ is \emph{tame}, 
  if all eigenvalues of all matrices $A \in \cS$ are contained in $\{0\} \cup \ru$.
\end{definition}

We are now ready to state the main result of this section. 
We already know that a $k$-Mahler function with $h(a_n)  \in O(\log n)$ is $k$-regular. 
The sequence of coefficients $(a_n)_{n\ge 0}$ therefore has a minimal linear representation 
$(u,\mu,v)$ by \cref{t:regular-linrep}.  

\begin{theorem} \label{t-main:hlog}
  Let $f(z) = \sum_{n=0}^\infty a_n z^n \in \power{\Qbar}$ be a $k$-Mahler function.
  The following statements are equivalent.
  \begin{equivenumerate}
  \item \label{hlog:hlog} We have $h(a_n) \in o(\log n)$.
  \item \label{hlog:tame} For every minimal 
  linear representation $(u, \mu, v)$ of $(a_n)_{n \ge 0}$, the matrix semigroup $\mu(\digits^*)$ is tame. 
  \item \label{hlog:convolution} The sequence $(a_n)_{n \ge 0}$ is a $\Qbar$-linear combination of 
  word-convolution products of $k$-automatic sequences.
   \item \label{hlog:hloglog} We have $h(a_n) \in O(\log\log n)$.
  \end{equivenumerate}
\end{theorem}

Our first goal is to use the additional restriction $h(a_n) \in o(\log n)$ to obtain a restriction 
on the possible eigenvalues of the matrices $\mu(w)$. 
The following lemma is similar to \cite[Lemma 2.3]{bell05}.

\begin{lemma} \label{l:tame}
  Let $(a_n)_{n \ge 0}$ be a $k$-regular sequence in $\Qbar$, with a minimal linear representation 
  $(u,\mu,v)$.
  Suppose $h(a_n) \in o(\log n)$.
  Then the finitely generated matrix semigroup $\mu(\digits^*)$ is tame.
\end{lemma}

\begin{proof}
  By definition of the linear representation, we have $a_{[w]_k} = u \mu(w) v$ for all $w \in \digits^*$. 
  Since $[w]_k \in O(k^{\length{w}})$ for all $w \in \digits^*$, our assumption on the sequence translates into 
  $h(u \mu(w) v) \in o(\length{w})$.
  
  Write $d$ for the dimension of $(u,\mu,v)$.
  If $d=0$, the claim is trivially true.
  Let $d > 0$.
  Suppose there exist a word $w \in \digits^*$ and $\lambda \in \Qbar\setminus\{0\}$ 
  not a root of unity such that $\lambda$ is an eigenvalue of $\mu(w)$. 
  Then there exists  a non-zero vector $v_0 \in \Qbar^{d \times 1}$ with $\mu(w)v_0 = \lambda v_0$.
  By minimality of the linear representation, there exist $w_1$, $\ldots\,$,~$w_d \in \digits^*$ such that $\mu(w_1)v$, $\ldots\,$,~$\mu(w_d) v$ form a basis of $\Qbar^{d \times 1}$.
  Let $\alpha_1$, $\ldots\,$,~$\alpha_d \in \Qbar$ be such that $v_0 = \alpha_1 \mu(w_1) v + \cdots + \alpha_d \mu(w_d) v$.
  Again by minimality, the set $\{\, u \mu(w') : w' \in \digits^* \,\}$ spans $\Qbar^{1 \times d}$.
  Therefore there exists $w' \in \digits^*$ such that $u \mu(w') v_0 \ne 0$.

  Now
  \[
    \sum_{i=1}^d \alpha_i u \mu(w'w^nw_i) v =  u \mu(w')\mu(w^n)v_0 = \lambda^n u \mu(w')v_0 \ne 0.
  \]
  Since $\lambda$ is not a root of unity, there exists an absolute value $\abs{\cdot}$ on $\Qbar$ with $\abs{\lambda} > 1$.
  We conclude that there exists an $i \in \{1,\ldots\,d\}$ with
  \[
    \abs{\alpha_i u\mu(w'w^nw_i)v} \ge \abs{\lambda}^n \cdot \frac{\abs{u \mu(w')v_0}}{d}
  \]
  for infinitely many $n \ge 0$.
  Hence there exists $c' \in \bR_{>0}$ such that, for these $n \ge 0$,
  \[
    h(u \mu(w'w^nw_i)v) \ge \log \abs{u \mu(w'w^nw_i)v} \ge n c'.
  \]
  This is a contradiction.
\end{proof}

Tame semigroups afford a particular block diagonal decomposition.

\begin{lemma} \label{l:structure-tame}
  Let $\cS \subseteq \Qbar^{d \times d}$ be a finitely generated tame semigroup.
  Then there exist $d_1$, $\ldots$,~$d_r \in \bZ_{\ge 0}$ with $d = d_1 + \cdots + d_r$, finite 
  semigroups $\cS_i \in \Qbar^{d_i \times d_i}$ for $i \in \{1,\ldots, r\}$, and a matrix $T \in \GL_d(\Qbar)$ 
  such that
  \[
    T^{-1} \cS T \subseteq
    \begin{pmatrix}
      \cS_1    & \Qbar^{d_1 \times d_2} & \Qbar^{d_1 \times d_3} & \dots   & \Qbar^{d_1\times d_r}  \\
      0      & \cS_2                  & \Qbar^{d_2 \times d_3} & \dots   & \Qbar^{d_2\times d_r}  \\
      0      & 0                    & \cS_3                  & \dots   & \Qbar^{d_3\times d_r}  \\
      \vdots & \vdots               & \vdots               & \ddots  & \vdots               \\
      0      & 0                    & 0                    & \dots   & \cS_{r}                \\ 
    \end{pmatrix}\, .
  \]
\end{lemma}

\begin{proof}
  If $\cS$ spans $\Qbar^{d \times d}$, then  $\mu(\digits^*)$ is finite \cite[Lemma 4]{bell-coons-hare16} 
  and we are done.  Otherwise, we iterate Lemma 5 of \cite{bell-coons-hare16} to get a 
  block-upper-triangular decomposition with finite semigroup diagonals.
\end{proof}

The arguments in the following proof are similar to \cite[Theorem 2.6]{bell05}.

\begin{proof}[Proof of \cref{t-main:hlog}]
  Let $f(z) = \sum_{n=0}^\infty a_n z^n \in \power{\Qbar}$ be $k$-Mahler.
  
  \ref{hlog:hlog}$\,\Rightarrow\,$\ref{hlog:tame}
  Suppose $h(a_n) \in o(\log n)$.
  Then $(a_n)_{n \ge 0}$ is $k$-regular by \cref{t-main:hlog2}.
  \Cref{l:tame} implies that $\mu(\digits^*)$ is tame.
  
  \ref{hlog:tame}$\,\Rightarrow\,$\ref{hlog:hloglog}
  Let $(u,\mu,v)$ be a minimal linear representation of the $k$-regular sequence $(a_n)_{n \ge 0}$.
  Suppose $\mu(\digits^*)$ is tame.
  We have to show $h(a_n) \in O(\log\log n)$.
  For this, it suffices to show $h(u \mu(w) v) \in O(\log \length{w})$ for non-empty words $w \in \digits^*$.
  We can apply \cref{l:structure-tame}.
  Thus, there exists a finite semigroup $\cS$ of block-diagonal matrices such that, after a change of basis, 
  for every $w \in \digits^*$ the matrix $\mu(w)$ is of the form $D + N$ with $D \in \cS$ and $N$ 
  strictly upper triangular. We may assume that $\cS$ contains the identity matrix.
  Since $\digits$ is finite, there exists a finite set $\cN$ of strictly upper triangular matrices such 
  that $\mu(\digits) \subseteq \cS + \cN$.

  Let $w = a_1 \cdots a_l \in \digits^*$ with $a_1$, $\ldots\,$,~$a_l \in \digits$, and let 
  $\mu(a_i) = D_i + N_i$ with $D_i \in \cS$ and $N_i \in \cN$.
  For $J \subseteq \{1,\ldots,l\}$ with $J = \{j_1 < j_2 < \cdots < j_r \}$ define
  \[
    b_J = uD_{1} \cdots D_{j_1-1} N_{j_1} D_{j_1+1} \cdots  D_{j_2-1} N_{j_2} D_{j_2+1} 
    \cdots D_{j_r-1} N_{j_r} D_{j_r+1} \cdots D_{l}v \,.
  \]
  Then
  \[
    u\mu(w) v = u \mu(a_1\cdots a_l) v = u (D_1 + N_1)\cdots (D_l + N_l) v 
    = \sum_{J \subseteq \{1,\ldots,l\}} b_J\,.
  \]
  Any product that includes $d$ or more of the $N_i$'s is $0$, and hence 
  $b_J=0$ whenever $\card{J} \ge d$. Thus, the previous sum reduces to
  \[
    u\mu(w) v = \sum_{\substack{J \subseteq \{1,\ldots,l\} \\ \card{J} < d}} b_J\,.
  \]
  This sum has at most $\binom{l}{d-1} + \cdots + \binom{l}{0} \le C l^{d-1}$ 
  non-zero terms for some constant $C \in \bR_{>1}$.
  As $\cS$ is a semigroup, each product $D_{j_i+1} \cdots D_{j_{i+1}-1}$ is 
  again contained in the finite set $\cS$. Hence
  \[
    \card{ \big\{\, b_{J} : J \subseteq \{1,\ldots,l\} \,\big\} } \le d\, \card{\mathcal S}^{d}\, 
    \card{\mathcal N}^{d-1} < \infty\,.
  \]

  Let $K$ be the number field generated by the finitely many coordinates of $u$, $v$, 
  and $\mu(a)$ for $a \in \digits$. Then $b_J \in K$ for each $J \subseteq \{1,\ldots,l\}$.
  Since there are only finitely many of these elements, for every place $\sv$ of $K$, 
  there exists a constant $c_\sv \in \bR_{>0}$ such that $\abs{b_{J}}_\sv \leq c_\sv$ for all 
  $J \subseteq \{1,\ldots,l\}$, and we can take $c_\sv = 1$ for all but finitely many places.
  For $m \in \bZ_{\ge 0}$, let
  \[
    \varepsilon_\sv(m) =
    \begin{cases}
      m &\text{if $\sv$ is archimedean,}\\
      1 &\text{if $\sv$ is non-archimedean.}
    \end{cases}
  \]
  Note $\prod_{\sv \in M_K} \varepsilon_{\sv}(m) \le m^{[K:\bQ]}$.
  With this definition $\abs{u \mu(w) v}_\sv \le \varepsilon_\sv(Cl^{d-1})c_\sv$ and
 
 \begin{eqnarray*}
      h(u\mu(w)v) &=& \log \prod_{\sv \in M_K} \max\{ 1, \abs{u \mu(w) v}_\sv \} \\
      & \le& \log \prod_{\sv \in M_K} \max\{1, \varepsilon_\sv(C l^{d-1})c_\sv \} \\
                  &\le &\log\big( (C l^{d-1})^{[K:\bQ]} \big) + \log\bigg( \prod_{\sv \in M_K} \max\{1,c_\sv\} \bigg) \in O(\log l).
  \end{eqnarray*}
  Since $l = \length{w}$ this proves the claim.

  \ref{hlog:hloglog}$\,\Rightarrow\,$\ref{hlog:hlog}
  Clearly $h(a_n) \in O(\log\log n)$ implies $h(a_n) \in o(\log n)$.

  Finally, the equivalence \ref{hlog:tame}$\,\Leftrightarrow\,$\ref{hlog:convolution} is precisely the equivalence (i)$\,\Leftrightarrow\,$(ii) proved by Bell, Coons, and Hare in \cite[Theorem 13]{bell-coons-hare16} (which does not require the sequences to be $\bZ$-valued).
\end{proof}

\section{Fourth gap: characterization of automatic Mahler functions}
\label{sec:hloglog}

In this section, we characterize Mahler functions $f(z)= \sum_{n=0}^\infty a_n z^n\in \power{\Qbar}$ 
with  $h(a_n) \in o(\log\log n)$, extending \cite[Theorem 1.1]{bell-coons-hare14}.

\begin{theorem} \label{t-main:hloglog}
  Let $f(z)= \sum_{n=0}^\infty a_n z^n \in \power{\Qbar}$ be a $k$-Mahler function.
  The following statements are equivalent.
  \begin{equivenumerate}
  \item \label{hloglog:hloglog} We have $h(a_n) \in o(\log\log n)$.
  \item \label{hloglog:semigroupfinite} For every minimal linear representation  of 
  $(a_n)_{n \ge 0}$, the matrix semigroup $\mu(\digits^*)$ is finite.
  \item \label{hloglog:automatic} The power series $f(z)$ is $k$-automatic.
  \item \label{hloglog:finite} We have $h(a_n)\in O(1)$. Equivalently, the set $\{\, a_n : n \ge 0 \,\}$ is finite.
  \end{equivenumerate}
\end{theorem}

The following lemma closely follows \cite[Lemma 2.1]{bell-coons-hare14}
\begin{lemma}  \label{l:semigroup-finite}
  Let $(a_n)_{n \ge 0}$ be a $k$-regular sequence in $\Qbar$, with a minimal linear representation 
  $(u,\mu,v)$. If $h(a_n) \in o(\log\log n)$, then the semigroup $\mu(\digits^*)$ is finite.
\end{lemma}

\begin{proof}
  Again $a_{[w]_k} = u \mu(w) v$ for all $w \in \digits^*$.
  By our assumption
  \[
    h(u \mu(w) v) \in o(\log \length{w})\,.
  \]
  
  Now suppose to the contrary that $\mu(\digits^*)$ is infinite.
  A theorem of McNaughton and Zalcstein \cite{mcnaughton-zalcstein75} gives a positive answer 
  to the strong Burnside problem for semigroups of matrices over a field.
  Since $\mu(\digits^*)$ is a finitely generated semigroup of matrices, but not finite, 
  this theorem implies that there exists $w \in \digits^*$ such that $\mu(w^m) \ne \mu(w^n)$ 
  for all $m$,~$n \in \bZ_{\ge 0}$ with $m \ne n$.
  Fix such a word $w$.

  Set $A\coloneqq \mu(w)$.
  By \cref{l:tame} every eigenvalue of $A$ is either $0$ or a root of unity. Our choice of $w$ 
  ensures that there exists at least one non-zero eigenvalue $\zeta$ with a non-trivial Jordan block.
  Let $B \in \Qbar^{d \times d}$ be an invertible matrix such that $B^{-1}AB$ is in Jordan normal form.
  Without restriction we may assume
  \[
    B^{-1}AB =
    \begin{pmatrix}
      \zeta & 1 & 0 & \dots & 0 \\
      0 & \zeta & * & \dots & 0 \\
      0 & 0 & * & \dots & 0 \\
      \vdots & \vdots & \vdots & \ddots & \vdots \\
      0 & 0 & \dots & \dots & *  \\
    \end{pmatrix}\,.
  \]
  The $(1,2)$ entry of $B^{-1}A^nB$ is $n\zeta^{n-1}$.
  Hence $h(e_1^T B^{-1}A^nBe_2) = h(n \zeta^{n-1}) \ge \log n$.

  Using the minimality of the linear representation, we can write 
  $e_1^TB^{-1} = \sum_{i=1}^d \lambda_i u \mu(w_i)$ 
  and $Be_2 = \sum_{i=1}^d \tau_i \mu(w_i') v$ with suitable 
  $\lambda_i$,~$\tau_i \in \Qbar$ and $w_i$,~$w_i' \in \digits^*$.
  It follows that
  \[
    h\Big( \sum_{i,j=1}^d \lambda_i \tau_j u \mu(w_i w^n w_j') v \Big)  \ge \log n\,.
  \]
  Hence there exist $i$,~$j \in \{1,\ldots,d\}$ and $c \in \bR_{>0}$ such that
  \[
    h( u\mu(w_i w^n w_j')v ) > c \log n \qquad\text{for infinitely many $n$.}
  \]
  This is a contradiction to $h(u\mu(w_i w^n w_j')v) \in o(\log n)$.
\end{proof}

\begin{proof}[Proof of \cref{t-main:hloglog}]
  Let $f(z) = \sum_{n=0}^\infty a_n z^n \in \power{\Qbar}$ be a $k$-Mahler function.
 
  \ref{hloglog:hloglog}$\,\Rightarrow\,$\ref{hloglog:semigroupfinite}
  Let $h(a_n) \in o(\log\log n)$.
  Then $(a_n)_{n \ge 0}$ is $k$-regular by \cref{t-main:hlog2}, with a minimal linear representation $(u,\mu,v)$.
  By \cref{l:semigroup-finite} the semigroup $\mu(\digits^*)$ is finite.

  \ref{hloglog:semigroupfinite}$\,\Rightarrow\,$\ref{hloglog:automatic}
  Let $(u,\mu,v)$ be a minimal linear representation of the regular sequence $(a_n)_{n \ge 0}$.
  Suppose $\mu(\digits^*)$ is finite.
  Then $(a_n)_{n \ge 0}$ takes only finitely many values, and $k$-regular sequences taking finitely many values are automatic (\cite[Theorem 16.1.5]{allouche-shallit03} or \cite[Proposition 5.3.3]{berstel-reutenauer11}).

  \ref{hloglog:automatic}$\,\Rightarrow\,$\ref{hloglog:finite}
  Let $(a_n)_{n \ge 0}$ be $k$-automatic.
  Then the sequence only takes finitely many values by definition.

  \ref{hloglog:finite}$\,\Rightarrow\,$\ref{hloglog:hloglog}
  Clearly, if $\{\, a_n : n \ge 0 \,\}$ is finite, then $h(a_n) \in o(\log\log n)$.
\end{proof}

\section{Comments on Becker's conjecture} \label{sec: becker}

Every $k$-regular power series is $k$-Mahler, and as a partial converse Becker showed 
that a $k$-Becker power series is $k$-regular.
He also conjectured a full description of $k$-regular power series in terms of $k$-Becker power series.
This conjecture was recently proven by Bell, Chyzak, Coons, and Dumas, as the main result in 
\cite{bell-chyzak-coons-dumas18}. The proof in \cite{bell-chyzak-coons-dumas18} is stated for $K=\mathbb C$, 
but the same arguments apply equally well to arbitrary fields of characteristic zero. 

\begin{theorem}[{\cite[Theorem 1]{bell-chyzak-coons-dumas18}}]\label{thm: becker}
  Let $K$ be a field of characteristic $0$.
  If $f(z) \in \power{K}$ is $k$-regular, there exist a polynomial $q(z) \in K[z]$ with $q(0)=1$ such 
  that $1/q(z)$ is $k$-regular and a non-negative integer $\gamma$ such $f(z)/z^\gamma q(z)$ is a $k$-Becker 
  Laurent series.
\end{theorem}

By a \emph{$k$-Becker Laurent series}, we of course mean a Laurent series satisfying a functional equation as in \cref{d:becker}.
We stress that it is not always possible to obtain a $k$-Becker power series of the form $f(z)r(z)$ with $r(z)$ a rational function \cite[Theorem 14]{bell-chyzak-coons-dumas18}.

The proof of Bell, Chyzak, Coons, and Dumas breaks down into two steps:
\begin{enumerate}[label=(\Roman*)]
  \item First they show that a $k$-regular power series 
    $f(z) \in \power{K}$ satisfies a $k$-Mahler equation
    \[
      p_0(z) f(z) + p_1(z) f(z^k) + \cdots + p_d(z) f(z^{k^d}) =0\,,
    \]
    where all roots of $p_0(z)$ belong to $\{0\} \cup \ru_k$. Equivalently, the $k$-denominator 
    $\mathfrak d(z)$ of $f(z)$ has all its roots in $\{0\} \cup \ru_k$. 
  \item They show that any such series has the required decomposition.
\end{enumerate}

We now give alternative arguments for both of these steps using our results.
In particular, for $K=\Qbar$, step I is immediate from \cref{t-main:hlog2} and our argument for step II is somewhat shorter.
We also recover Proposition 2 and Corollary 3 of \cite{bell-chyzak-coons-dumas18} (Corollary 3 follows as in the proof of \cref{p:hlog2}). 

\subsection{Step I}
For $K = \Qbar$, \cref{t-main:hlog2} immediately establishes step I.
We now show how to extend the relevant part of \cref{t-main:hlog2} to arbitrary fields of characteristic $0$.

Let $K$ be a field of characteristic $0$ and let $f(z)=\sum_{n=0}^\infty a_n z^n \in \power{K}$ be $k$-regular.
Let $\mathfrak d(z)$ denote the Mahler denominator of $f(z)$ over $K$.
We also have that the $k$-kernel of $f(z)$ generates a finite-dimensional $K$-vector space.
In particular, there is some fixed $M > 0$ such that for every $j\in \{0,1,\ldots ,k^M-1\}$, we have
\[
  a_{k^M n + j} = \sum_{e<M} \sum_{i=0}^{k^e-1} c_{j,i,e} a_{k^e n+i} \qquad\text{for $n \ge 0$.}
\]
We have that for some $s \ge 0$, the power series $f(z),f(z^k),\ldots,f(z^{k^s})$ satisfy a Mahler system of the form \cref{eq:mahler-lin} with some invertible matrix $A(z)$ with entries in $K(z)$.
Let $R$ be a finitely generated $\mathbb{Z}$-algebra that contains:
\begin{enumerate}
\item $a_0,\ldots,a_{k^M-1}$;
\item the roots of $\mathfrak d(z)$ and the reciprocals of all non-zero roots of $\mathfrak d(z)$;
\item the structure constants $c_{i,j,e}$.
\item the non-zero coefficients and their inverses of each polynomial appearing in either the numerator or denominator of an entry of $A(z)$.
\end{enumerate}
Then by construction, $f(z) \in \power{R}$.
If $\mathfrak p$ is a prime ideal of $R$ and $g(z)\in \power{R}$, then we let $g_{\mid \mathfrak p}(z)$ denote the power series in $\power{R/\mathfrak p}$ obtained by reducing the coefficients of $g(z)$ modulo $\mathfrak p$.
Then by construction, $f_{\mid \mathfrak p}(z)$ is a regular power series in $\power{(R_\mathfrak p/\mathfrak p_\mathfrak p)}$.

\begin{lemma}
  Let $\lambda\in R$ be a non-zero root of $\mathfrak d(z)$.
  Then $\lambda$ is a root of unity.  
\end{lemma}

\begin{proof}
  Suppose that $\lambda$ is not root of unity.
  Then, since the coefficients of $A(z)$ have only finitely many poles, there is some $n$ such that $\lambda^{k^{n}}$ is a regular point with respect to this Mahler system.
  We now take a $k$-Mahler equation
  \begin{equation} \label{eq:kk-min-k-mahler}
    q_0(z) f(z) = \sum_{i=1}^L q_i(z) f(z^{k^i}), \qquad q_0(z), q_1(z),~\ldots\,,~q_L(z) \in K[z],
  \end{equation}
  with $q_0(z) \ne 0$ and $L$ minimal.
  Iterating \cref{eq:kk-min-k-mahler} we find an equation
  \begin{equation} \label{eq:kk-iterated}
    r_0(z) f(z) = \sum_{i=n}^{L + n-1} r_i(z) f(z^{k^i}),
  \end{equation}
  where we may assume that the polynomials $r_0(z),r_{n}(z),\ldots,r_{L+n+1}(z) \in K[z]$ are coprime.
  The Mahler denominator $\mathfrak d(z)$ divides $r_0(z)$ and so $r_0(\lambda) = 0$.
  By coprimality of the coefficients, in particular there is some $i_0$ such that $r_{i_0}(\lambda)\neq 0$.
    Now we adjoin the coefficients of $r_0(z)$ and $r_n(z),\ldots,r_{n+L-1}(z)$ to $R$.
  
    By Noether normalization, there is a positive integer $N$ and $x_1,\ldots,x_d \in R$ such that 
    $x_1,\ldots,x_d$ are algebraically independent over $\mathbb Q$ and such that 
    $R[1/N]$ is a finite integral extension of $\mathbb{Z}[1/N][x_1,\ldots, x_d]$.
    Let $\mathcal{S}$ denote the set of prime ideals $\mathfrak p$ of $R[1/N]$ with the property that 
    $\mathfrak p\cap \mathbb{Z}[1/N][x_1,\ldots ,x_d] = (x_1-b_1,\ldots ,x_d-b_d)$ with $b_1, \ldots,b_d$ integers.
    By integrality, there is at least one such prime for each $d$-tuple $(b_1,\ldots, b_d)$ of integers.
    Moreover, $R/\mathfrak p$ is a finite extension of $\mathbb{Z}[1/N]$, generated by at most $\kappa$ elements for some $\kappa$ that is independent of $\mathfrak p$ (indeed, we may take $\kappa$ to be the cardinality of the set of generators of $R[1/N]$ as a $\mathbb{Z}[1/N][x_1,\ldots ,x_d]$-module).
    Hence $R_\mathfrak p/\mathfrak p_\mathfrak p$ is a number field of degree at most $\kappa$ for each $\mathfrak p\in \mathcal{S}$.
    
    Moreover, the intersection of the prime ideals in $\mathcal{S}$ is $(0)$.  
    For $\mathfrak p\in \mathcal{S}$, we let $\lambda_\mathfrak p=\lambda+\mathfrak p\in R/\mathfrak p$.
    Then we reduce \cref{eq:kk-iterated} modulo $\mathfrak p$ and plug in $z=\lambda_\mathfrak p$ to obtain
\begin{equation}\label{eq: 0}
0 =\sum_{i=n}^{L+n-1} r_{i\mid \mathfrak p}(\lambda_\mathfrak p) f_{\mid \mathfrak p}(\lambda_\mathfrak p^{k^{i}}),
\end{equation}
where the left side follows from the fact that $\mathfrak d(z)$ divides $r_0(z)$.
It is straightforward to see that $\lambda_\mathfrak p^{k^{n}}\in (R/\mathfrak p)_\mathfrak p$ is a regular point of the reduced Mahler system for $f_{\mid \mathfrak p}(z),f_{\mid \mathfrak p}(z^k),\ldots,f_{\mid \mathfrak p}(z^{k^{L-1}})$ for 
$\mathfrak p$ in a Zariski dense subset $\mathcal{T}$ of $\mathcal{S}$.
Moreover, there is a Zariski dense subset $\mathcal{T}'$ of $\mathcal{T}$ such that $r_{i_0\mid \mathfrak p}(\lambda_\mathfrak p)\neq 0$ for $\mathfrak p\in \mathcal{T}'$. 
We remark that there is a Zariski dense subset $\mathcal{T}''$ such that $\lambda_\mathfrak p$ is not a root of unity.
To see this, observe that if $\mathfrak p\in \mathcal{T}'$ is such that $\lambda_\mathfrak p$ is a root of unity, then $\mathbb{Q}(\lambda_\mathfrak p)$ is an extension of degree at most $\kappa$ and hence there is some fixed $M=M(\kappa)$ such that $\lambda_\mathfrak p^M=1$.
Since $\mathcal{T}'$ is Zariski dense, this gives that $\lambda$ is a root of unity, which is a contradiction.

Now for $\mathfrak p\in \mathcal{T}''$, we have $\lambda_\mathfrak p\in K_\mathfrak p:=(R/\mathfrak p)_\mathfrak p$.
Then there is some place $\sv$ on the number field $K_\mathfrak p$ such that $|\lambda_\mathfrak p|_{\sv} <1$.
Equation \eqref{eq: 0} combined with \cref{thm: lifting} yields
\[
  0 = \sum_{i=n}^{L+n-1} r_{i\mid \mathfrak p}(z) f_{\mid \mathfrak p}(z^{k^i}).
\]
By Zariski density of $\mathcal{T}''$, also $0 =\sum_{i=n}^{L+n-1} r_i(z) f(z^{k^i}) \in \power{R}  \cap \power{K}$, in contradiction to the minimality of $L$.
The result follows. 
\end{proof}

By looking at the asymptotic behavior on the unit circle we may once again strengthen the previous lemma.

\begin{lemma}
  Let $\lambda\in R$ be a non-zero root of $\mathfrak d(z)$.
  Then $\lambda \in \rucommon$.
\end{lemma}

\begin{proof}
  Let $\lambda$ be a non-zero root of $\mathfrak d(z)$.
  By the previous lemma $\lambda \in \ru$.
  Therefore it suffices to show $\lambda^{k^j} \ne \lambda$ for all $j \ge 1$.

  Suppose to the contrary that $\lambda^{k^{j}} = \lambda$ for some $j \ge 1$.
  Since $R$ is finitely generated, it embeds into $\mathbb C$.
  Let $\abs{\cdot}$ denote the induced absolute value on $R$.
  We may now conclude as in the proof of \cref{p:goodroots}:
  from \cref{p:coprime-k-lower} we obtain $\log \abs{a_n} \ge c \log^2 n$ infinitely often.
  However, using the linear representation of a $k$-regular sequence, we easily obtain $\log \abs{a_n} \in O(\log n)$ as in \cref{l:growth-regular} (or \cite[Theorem 2.10]{allouche-shallit92}), a contradiction.
\end{proof}

We thus have the following theorem, extending a part of \cref{t-main:hlog2} to fields of characteristic $0$ and also generalizing \cite[Proposition 2]{bell-chyzak-coons-dumas18}.

\begin{theorem}
  Let $K$ be a field of characteristic $0$, let $f(z) \in \power{K}$ be $k$-Mahler, and let $\mathfrak d(z)$ be the Mahler denominator of $f(z)$.
  Then $f(z)$ is $k$-regular if and only if every non-zero root of $\mathfrak d(z)$ \textup{(}in the algebraic closure $\overline{K}$\textup{)} is a root of unity with order not coprime to $k$.
\end{theorem}

\begin{proof}
  If $f(z)$ is $k$-regular, the claim follows from the previous lemma.
  The converse direction follows exactly as in the proof of \cref{p:hlog2}.
\end{proof}

\subsection{Step II}
We now provide a somewhat shorter argument for the second step of \cite{bell-chyzak-coons-dumas18}.
First recall the following easy lemma. 

\begin{lemma}\label{lem: easy}
Let $K$ be a field and let $f(z)\in \power{K}$ be a $k$-Mahler power series solution to the equation 
\begin{equation}\label{eq1}
p_0(z)f(z)+p_1(z)f(z^k)+\cdots +p_d(z)f(z^{k^d})=0 \,.
\end{equation}
If there exists a polynomial $q(z)$ such that $p_0(z)q(z)$ divides $q(z^{k^j})$ for all $j$, $1\leq j\leq d$, then 
$f(z)/q(z) \in \laurent{K}$ is a $k$-Becker Laurent series. 
\end{lemma}

\begin{proof}
Set $g(z):=f(z)/q(z)$, then \eqref{eq1} gives 
$$
p_0(z)q(z)g(z) +p_1(z)q(z^k)g(z^k)+\cdots +p_d(z)q(z^{k^d}) g(z^{k^d})=0 \,.
$$
Thus, $g(z) = -\sum_{i=1}^d r_i(z)g(z^{k^i})$, where   
$r_i(z)=p_i(z)q(z^{k^i})/(p_0(z)q(z))\in K[z]$.  
\end{proof}

\begin{proof}[Proof of Becker's conjecture \textup{(}\cref{thm: becker}\textup{)}]
Since $f(z)$ is $k$-regular, we know, by the first step, that 
$f(z)$ satisfies an equation of the form 
$$
p_0(z)f(z)+ p_1(z)f(z^k)+\cdots + p_d(z)f(z^{k^d})=0 \,,
$$
where all roots of $p_0(z)$ belong to $\{0\}\cup \ru_k$. Thus, every non-zero root is  a 
primitive $\ell$-root of unity  for some $\ell$ not coprime with $k$. 
For such a natural number $\ell$, there exist a positive integer $r$ and a non-negative integer $s$ such that 
$\gcd(\ell,k^j)=r$ for all $j> s$. Let $s$ be minimal with this property.  
Let $A$ denote the set of non-zero roots of $p_0(z)$, and, for $\xi\in A$, set 
$a(\xi):=\ell(\xi)/r(\xi)$.  Let $\phi_n(z)$ denote the $n$th cyclotomic polynomial. Then 
$\phi_{\ell(\xi)}(z)\phi_{a(\xi)}(z^{k^s})$ divides $\phi_{a(\xi)}(z^{k^{s+j}})$ for all $j\geq 1$.  
In particular, $(z-\xi)\phi_{a(\xi)}(z^{k^s})$ divides $\phi_{a(\xi)}(z^{k^{s+j}})$ for all $j$, $1\leq j\leq m$. 
Setting 
$$
q(z):= \prod_{\xi\in A}\phi_{a(\xi)}(z^{k^s})\,,
$$ 
and applying Lemma \ref{lem: easy}, we obtain that $f(z)/z^{\gamma}q(z)$ is a $k$-Becker Laurent series, 
where $\gamma$ is the valuation of $p_0(z)$. 
Moreover, $1/q(z)$ is a $k$-Becker power series for 
$$
q(z)^{-1}= \frac{q(z^k)}{q(z)} \cdot q(z^k)^{-1}
$$
and by construction $q(z)$ divides $q(z^k)$.  In particular, $1/q(z)$ is $k$-regular. 
\end{proof}


\section{Automatic Mahler power series over arbitrary fields}\label{sec: automatic}

We first show how to extend our characterization of $k$-automatic Mahler functions to arbitrary 
ground fields 
of characteristic zero.

\begin{theorem} \label{c:finite-automatic}
  Let $K$ be a field of characteristic $0$ and let $f(z)= \sum_{n=0}^\infty a_n z^n\in\power{K}$ 
be a $k$-Mahler power series. 
  Then $(a_n)_{n\ge 0}$ is $k$-automatic if and only if   $\{\, a_n : n \ge 0 \,\}$ is finite.  
\end{theorem}

In order  to prove \cref{c:finite-automatic}, we use a standard specialization argument.

\begin{lemma} \label{l:specialization}
  Let $K$ be a field of characteristic zero containing $\Qbar$, and let $u_1$, $\ldots\,$,~$u_d \in K$.
  Then there exists a ring homomorphism $\varphi \colon \Qbar[u_1,\ldots,u_d] \to \Qbar$ leaving $\Qbar$ invariant.
\end{lemma}

\begin{proof}
  This is an easy consequence of  the weak Nullstellensatz.
  A proof can be found in \cite[Lemma 6.3.3]{evertse-gyory15}.
\end{proof}

\begin{proof}[Proof of Theorem \ref{c:finite-automatic}]
  Let $f(z) = \sum_{n=0}^\infty a_n z^n \in \power{K}$ be a $k$-Mahler power series 
  with finite set of coefficients. Replacing $K$ by its algebraic closure, we may assume 
  $\Qbar \subseteq K$.
  Let $p_0(z),\ldots,p_d(z) \in K[z]$ be such that
  \[
    p_0(z) f(z) + p_1(z) f(z^k) + \cdots + p_d(z) f(z^{k^d})=0 \,.
  \]
  Let $C$ be the finite set consisting of all coefficients of $f(z)$ and $p_0(z),\ldots,p_d(z)$. 
  We apply \cref{l:specialization} with the set $\{u_1,\ldots,u_d\}$ consisting of all 
  $c \in C$, all $c -d$ with $c,d \in C$, as well as the inverses of all these elements that are non-zero.
  Thus $\varphi(c) \ne 0$ for $c \ne 0$ and $\varphi(c) \ne \varphi(d)$ for $c \ne d$.
  The resulting homomorphism extends to $\varphi\colon \power{\Qbar[u_1,\ldots,u_m]} \to \power{\Qbar}$, and
  \[
    \varphi(p_0)(z)\cdot \varphi(f)(z) + \varphi(p_1)(z) \cdot \varphi(f)(z^k) + \cdots + \varphi(p_d)(z)\cdot \varphi(f)(z^{k^d})=0
  \]
  is a $k$-Mahler equation for $\varphi(f)(z)$.
  Thus \cref{t-main:hloglog} implies that the sequence $(\varphi(a_n))_{n \ge 0}$ is $k$-automatic.
  Since $\varphi \colon C \to \Qbar$ is injective, the same is true for $(a_n)_{n \ge 0}$.
\end{proof}

\subsection{The case of a base field of positive characteristic}

\Cref{c:finite-automatic} strongly depends on the characteristic of the field being zero.
If $K$ is a field of characteristic $p>0$, we still have a similar result for $p$-Mahler power series 
(see Proposition \ref{prop: charp}), but if $k$ is coprime to $p$ this is no longer true. Indeed, 
the series
  \[
     \prod_{i=0}^\infty (1 - z^{k^i})^{-1} \in \power{K}
  \]
  is not $k$-automatic by \cite[Proposition 1]{becker94}, despite being $k$-Mahler with 
  coefficients taking only finitely many values (because they belong to the prime field). 

\begin{proposition}\label{prop: charp}
Let $K$ be a field of characteristic $p$ and let $f(z)= \sum_{n=0}^\infty a_n z^n\in\power{K}$ 
be a $k$-Mahler power series where $k$ is a power of $p$. Then 
the sequence $(a_n)_{n\ge 0}$ is $k$-automatic 
 if and only if $\{\, a_n : n \ge 0 \,\}$ is finite. 
\end{proposition}

\begin{proof} 
Let $f(z)= \sum_{n=0}^\infty a_n z^n\in\power{K}$ be  $p^m$-Mahler for some positive integer $m$. 
Let us assume that  $\{\, a_n : n \ge 0 \,\}$ is finite. 
Let us consider a  non-trivial equation 
$$
p_0(z)f(z)+p_1(z) f(z^{p^m})+\cdots+ p_d(z) f(z^{p^{md}})  = 0 \, . 
$$   
We let $R$ denote the finitely generated $\mathbb F_p$-algebra generated by the 
coefficients $a_n$, the inverses of all non-zero differences $a_i-a_j$, and the coefficients of the 
polynomials $p_i(z)$, as well as the inverses of their non-zero coefficients.  

Let $\mathfrak M$ be some maximal ideal of $R$.  Then $R/\mathfrak M=\mathbb F_q$ 
with $q$ a power of $p$, say $q=p^\ell$. 
 Let $f_{\mid \mathfrak M}(z):= \sum_{n=0}^\infty (a_n\bmod \mathfrak M) z^n$ 
 denote the reduction of $f(z)$ modulo $\mathfrak M$. 
 Then $f_{\mid \mathfrak M}(z)$  is $p^m$-Mahler and hence it is also $p^{m\ell}$-Mahler by \cref{l:mahler-power}.   
 Thus, we deduce that  $f_{\mid \mathfrak M}(z)$ is algebraic over $\mathbb F_q(z)$.  By Christol's theorem (see \cite[Chapter 12]{allouche-shallit03}), the sequence 
 $(a_n \bmod \mathfrak M)_{n\geq 0}$ is $p$-automatic. But by definition of $R$, if $a_i\not= a_j$ then 
 $a_i \bmod \mathfrak M \not= a_j \bmod \mathfrak M$. Thus the sequence $(a_n)_{n\geq 0}$  
is also $p$-automatic, and hence  $p^m$-automatic.    
\end{proof}

\section{Decidability}\label{sec: decidability}

A $k$-Mahler function can be uniquely specified by the finite data consisting of a $k$-Mahler 
equation it satisfies and sufficiently many initial coefficients of the power series.
Therefore it is reasonable to ask whether, for a given $k$-Mahler function, it can be decided 
which of the five cases of \cref{thm: hgt} it falls into. However, we neither try to describe an efficient 
algorithm to perform this task, nor do we provide an upper bound for the complexity of the algorithm that 
could be extracted from what follows. 

\begin{theorem} \label{t:decidability}
  Let $f(z) = \sum_{n=0}^\infty a_n z^n \in \power{\Qbar}$ be a $k$-Mahler function 
  \textup{(}specified by a $k$-Mahler equation and sufficiently many initial coefficients\textup{)}.
  Then it is decidable which of the five growth classes in \cref{thm: hgt} the function $f$ falls into.
\end{theorem}

Let $f(z) = \sum_{n=0}^\infty a_n z^n \in \power{\Qbar}$ be $k$-Mahler.
As \cref{t-main:hn,t-main:hlog2} show, the minimal denominator $\mathfrak d(z) \in \Qbar[z]$ of $f(z)$ 
plays a crucial role in determining the growth class that $f(z)$ falls into: 
the growth depends on whether $\mathfrak d(z)$ has roots outside $\{0\} \cup \ru$, 
respectively outside $\{0\} \cup \rucommon$. 
This raises the question whether there is an effective way of deciding which of the three cases occurs.
Along similar lines, if $f(z)$ is $k$-regular, the question arises whether it is decidable into which 
of the three cases ($k$-regular, $\Qbar$-linear combination of word-convolution products of 
$k$-automatic sequences, 
and $k$-automatic) the coefficients of $f(z)$ falls. 
In this section, we establish that all these properties are decidable. 

Suppose $f(z)$ is specified by a $k$-Mahler equation and sufficiently many initial coefficients 
to determine the solution uniquely. 
Then we can compute any finite number of initial coefficients by recursion.
By work of Adamczewski and Faverjon \cite{adamczewski-faverjon18} it is possible to find a 
\emph{minimal} (homogeneous) $k$-Mahler equation, that is, polynomials 
$p_0(z),p_1(z),\ldots,p_d(z) \in K[z]$, where $K$ is a number field, 
\begin{equation} \label{eq:eff-min}
  p_0(z) f(z) + p_1(z) f(z^k) + \cdots + p_d(z) f(z^{k^d})=0\,,
\end{equation}
where $d$ is minimal and $p_0(z),\ldots,p_d(z)$ are coprime. 

By definition $\mathfrak d(z)$ divides $p_0(z)$.
It is tempting to hope that, to determine the types of roots of $\mathfrak d(z)$, 
it suffices to consider those of $p_0(z)$.
Unfortunately, this hope is thwarted by Example \ref{exm: denominator}. 
We can however still determine the types of roots of $\mathfrak d(z)$.

\begin{proposition} \label{p:decide-ru}
  There exists an algorithm to determine whether the $k$-Mahler denominator $\mathfrak d(z)$ of a 
  $k$-Mahler function $f(z)$ has a root outside $\{0\} \cup \ru$. 

  Moreover, if all roots of $\mathfrak d(z)$ are contained in $\{0\} \cup \ru$, then we can find an explicit 
  $k$-Mahler equation
  \[
    q_0(z) f(z) = q_1(z) f(z^{k^{n_0}}) + \cdots + q_d(z) f(z^{k^{n_0+d-1}})
  \]
  with $n_0 \ge 1$, with $q_0(z),\ldots,q_d(z) \in \Qbar[z]$ and all roots of $q_0(z)$ contained in $\{0\} \cup \ru$. 
\end{proposition}

\begin{proof}
  Let us consider the minimal equation \eqref{eq:eff-min}.
  By \cite{adamczewski-faverjon18}, this equation can be explicitly determined 
  (this is a variation of Algorithm 1.3 in \cite{adamczewski-faverjon18}). 
  We may assume that the number field $K$ contains all coefficients and roots of 
  $p_0(z),\ldots,p_d(z)$. 
  Set
  \[
    \mathcal S \coloneqq \{\,\lambda : p_0(\lambda)p_d(\lambda)=0 \,\}
  \]
  and 
  \[
    \rho=\min_{\sv\in \mathcal M_K}\{\,\min\{\, \vert \lambda\vert_\sv : p_0(\lambda)p_d(\lambda)=0\,\}\,\}\,.
  \]
  Now, let $n_0$ be the minimal positive integer such that $\vert \lambda^{k^{n_0}}\vert_{\sv}<\rho$ 
  for all $\lambda$ in $\mathcal S$ and all places $\sv$ such that $\vert \lambda\vert_{\sv}<1$ 
  (there are only a finite number of such places).
  The integer $n_0$ can be explicitly determined.
  By repeated substitution, we can explicitly determine an equation
  \begin{equation}\label{eq:eff-iterate}
    q_0(z)f(z) = q_1(z)f(z^{k^{n_0}})+\cdots+q_d(z)f(z^{k^{n_0+d-1}})\,,
  \end{equation}
  for $f(z)$. 
  Suppose first that $q_0(z)$ does \emph{not} have a non-zero root $\lambda$ that is not a root of unity.
  Then neither does $\mathfrak d(z)$, because $\mathfrak d(z)$ divides $q_0(z)$.

  Suppose now $q_0(z)$ has a non-zero root $\lambda$ that is not a root of unity.
  By Kronecker's theorem, there exists a place $\sv$ such that $0<\abs{\lambda}_\sv<1$.
  Arguing exactly as in the proof of \cref{prop: radius}, we see that $\lambda$ is a pole of $f(z)$.
  Thus $f(z)$ has a radius of convergence strictly less than $1$ with respect to $\abs{\cdot}_\sv$. 
  By \cref{t-main:hn} also $\mathfrak d(z)$ must have a non-zero root that is not a root of unity. 
\end{proof}

Assuming $\mathfrak d(z)$ does not have a root outside of $\{0\} \cup \ru$, 
we now want to determine if it has a root in $\ru \setminus \rucommon$.

\begin{lemma}\label{l:upper-bound-matrix-system}
  Let $d \in \bZ_{>0}$, let $\zeta \in \bC\setminus\{0\}$ such that $\zeta^k = \zeta$, 
  and let $A(z) \in \Qbar(z)^{d \times d}$.
  Assume that $w(z) \in \power{\bC}^d$ satisfies the equation
  \[
    w(z) = A(z) w(z^k)\,.
  \]
  Assume also that the following properties hold.
  \begin{propenumerate}
  \item The coordinates of $A(z)$ have no poles at $\zeta$ and no poles in $B(0,1)$.
  \item The coordinates of $w(z)$ are continuous in $B(0,1)$.
  \end{propenumerate}
  Then, there exists $c \in \bR_{>0}$ such that
  \[
    \norm{w(t \zeta) } < \abs{1 - t}^{-c} \qquad \text{for all $t \in (0,1)$}.
  \]
\end{lemma}

\begin{proof}
  Since the map $[0,1] \to \bC^{d \times d}$, $t \mapsto A(t\zeta)$ is continuous, there exists $c_0 \ge 1$ such that $\norm{A(t \zeta)} \le c_0$ for all $t \in [0,1]$.
  Let $\varepsilon \in (0,1)$ and $c_1 = \max\{\, \norm{w(t \zeta)} : t \in [0, \varepsilon] \,\}$.

  Let $t \in [0,1)$, and let $n \in \bZ_{\ge 0}$ be minimal such that $t^{k^n} \le \varepsilon$.
  We can obtain an upper bound on $n$ as follows.
  The inequality $t^{k^n} \le \varepsilon$ is equivalent to $k^n \log t \le \log \varepsilon$, 
  which is equivalent to $k^n (-\log t) \ge -\log \varepsilon$.
  In turn, this is equivalent to $n + \log_k(-\log t) \ge \log_k(-\log \varepsilon)$.
  So
  \[
    n = \lceil \log_k(-\log \varepsilon) - \log_k(-\log t ) \rceil\,.
  \]
  Thus
  \[
    n \le c_2  - \log_k(-\log t) \qquad\text{with $c_2 = 1 + \log_k(-\log \varepsilon)$\,.}
  \]
  
  Now
  \[
    k^n \le k^{c_2} k^{-\log_k(-\log t)} \le k^{c_2} \frac{1}{-\log t} \le k^{c_2} \frac{1}{1-t}\,,
  \]
  where we used $\log t \le t - 1$ for the last inequality.
  We have
  \[
    w(t \zeta) = A(t \zeta) A(t^k \zeta) \cdots A(t^{k^{n-1}} \zeta) w(t^{k^n}\zeta)\,,
  \]
  and thus $\norm{w(t\zeta)} \le c_0^n c_1$.
  Now
  \[
    c_0^n c_1 = c_1 k^{n \log_k c_0} \le c_1  k^{c_2 \log_k c_0} (1-t)^{-\log_k c_0}.
  \]
  The constant may be absorbed by replacing the exponent by a bigger one.
\end{proof}

\begin{proposition} \label{p:decide-rucommon}
  Let $f(z) \in \power{\Qbar}$ be $k$-Mahler with $k$-Mahler denominator $\mathfrak d(z)$.
  Suppose all roots of $\mathfrak d(z)$ are contained in $ \{0\}\cup \ru$.
  There exists an algorithm to decide whether $\mathfrak d(z)$ has a root in $\ru \setminus \rucommon$.

  Moreover, if all roots of $\mathfrak d(z)$ are contained in $\{0\} \cup \rucommon$, 
  then we can find an explicit $k$-Mahler equation
  \[
    s_0(z) f(z) = s_1(z) f(z^{k}) + \cdots + s_d(z) f(z^{k^{d}})
  \]
  with $s_0(z),\ldots,s_d(z) \in \Qbar[z]$ and all roots of $s_0(z)$ contained in $\{0\} \cup \rucommon$.
\end{proposition}

\begin{proof}
  Let $\mathfrak d(z)f(z) = p_1(z) f(z^k) + \cdots + p_d(z) f(z^{k^d})$ with $p_1(z),\ldots,p_d(z) \in \Qbar[z]$.
  Using the condition on $\mathfrak d(z)$ together with the fact that $f(z)$ converges in a neighborhood of $0$, 
  this equation implies that $f(z)$ is analytic in $B_{\abs{\cdot}}(0,1)$ for every absolute value $\abs{\cdot}$ 
  on $\Qbar$.

  Now let $q_0(z) f(z) = q_1(z) f(z^k) + \cdots + q_d(z) f(z^{k^d})$ with $q_0(z),\ldots,q_d(z) \in \Qbar[z]$ 
  and $q_0(z)q_d(z) \ne 0$ be an explicit $k$-Mahler equation for $f(z)$. 
  Since $\mathfrak d(z)$ divides $q_0(z)$, we only have to check if any of the finitely many roots of $q_0(z)$ in 
  $\ru \setminus \rucommon$ are roots of $\mathfrak d(z)$.

  Suppose $\zeta$ is such a root of $q_0(z)$.
  Then there exists an, explicitly determinable, integer $j_0 \ge 1$ such that 
  $\zeta^{k^{j_0}} = \zeta$. Let $\ell = k^{j_0}$.
  Again using \cite{adamczewski-faverjon18} we can find an $\ell$-Mahler equation for $f(z)$, say
  \begin{equation} \label{eq:mahler-r}
    r_0(z) f(z) = r_1(z) f(z^\ell) + \cdots + r_e(z) f(z^{\ell^e})
  \end{equation}
  with $r_0(z),\ldots,r_e(z) \in \Qbar[z]$ coprime and $r_0(z) r_e(z) \ne 0$.
  If $r_0(\zeta) \ne 0$, then $\mathfrak d(\zeta) \ne 0$.

  Suppose now $r_0(\zeta) = 0$.
  We will show $\mathfrak d(\zeta) = 0$.
  Fix any embedding $\Qbar \hookrightarrow \bC$, and thereby an archimedean absolute value $\abs{\cdot}$ on $\Qbar$.
  \Cref{p:coprime-k-lower} implies that there exist $a,b,c \in \bR_{>0}$, $m_0,n_0 \in \bZ_{\ge 0}$, 
  and a sequence $(t_j)_{j \ge 0} \to 1$ in $[0,1)$ such that
  \[
    \abs[\big]{\sum_{n=0}^\infty a_{n+n_0} (\zeta t_j)^n } \ge (1-t_j)^a \exp(b \log^2m) t_j^{mc} 
    \qquad\text{for all $j \ge 0$ and $m \ge m_0$.}
  \]
  Define $f_0(z) = \sum_{n=0}^\infty a_{n+n_0} z^n$ and $m_j = \lceil 1/(1-t_j) \rceil$.
  Then
  \[
    \log \abs{f_0(\zeta t_j)} \ge a \log(1-t_j) +  b \log^2 \lceil 1/(1-t_j) \rceil + \lceil 1/(1-t_j) \rceil c \log t_j\,.
  \]
  As in the proof of \cref{l:estimate} we see that the right side is asymptotically equivalent to $b \log^2 (1-t_j)$. 
  Now, if we had $\mathfrak d(\zeta) \ne 0$, then \cref{l:upper-bound-matrix-system} would give 
  $\log\abs{f_0(\zeta t_j)} \le -c \log(1-t_j)$ for some $c \in \bR_{>0}$, a contradiction. 

  To explicitly find an equation with $s_0(z)$ as desired, note that for each 
  $\zeta \in \ru \setminus \rucommon$ that is a root of $q_0(z)$, we have found some 
  $k$-Mahler equation, \cref{eq:mahler-r}, for $f(z)$ with $r_0(\zeta) \ne 0$.
  Taking the greatest common divisor of $q_0(z)$ and all these $r_0(z)$ as $\zeta$ varies over the roots, 
  we obtain the desired equation.
\end{proof}

We have now shown that it is possible to decide algorithmically which of Cases \ref{tms:generic}  and 
\ref{tms:hn} 
of \cref{thm: hgt} a given Mahler function $f(z)=\sum_{n=0}^\infty a_n z^n \in \power{\Qbar}$ falls into.
Suppose now that $f(z)$ is $k$-regular. 
In this case, we wish to also decide whether $f(z)$ belongs to class \ref{tms:hlog2}, \ref{tms:hlog}, or 
\ref{tms:hloglog} of \cref{thm: hgt}. 

\subsection{From \texorpdfstring{$k$}{k}-Mahler equations to linear representations}

We have represented an arbitrary $k$-Mahler function $f(z)$ by a $k$-Mahler equation 
and sufficiently many initial coefficients.
If $f(z)$ is $k$-regular, it is more natural to represent the sequence of coefficients by a linear representation.
We show that such a linear representation is computable from a $k$-Mahler equation satisfied by $f(z)$. 
Recall that $\Delta_r$ with $r \in \digits$ denotes the Cartier operators (Definition~\ref{d:cartier}).

\begin{lemma} \label{l:relations-to-linrep}
  Let $f_1(z),\ldots,f_d(z) \in \power{\Qbar}$ with $f_i(z) = \sum_{n=0}^\infty a_{i,n} z^n$.
  Suppose that, for every $r \in \digits$ and every $1 \le i \le d$, there are explicitly known 
  coefficients $\lambda_{r,1},\ldots,\lambda_{r,d} \in \Qbar$ such that
  \[
    \Delta_r(f_i(z)) = \lambda_{r,1,i} f_1(z) + \cdots + \lambda_{r,d,i} f_d(z)\,.
  \]
  Then we get an explicit linear representation for the $k$-regular sequence $(a_{1,n})_{n \ge 0}$.
\end{lemma}

\begin{proof}
  Let $\mu \colon \digits^* \to \Qbar^{d \times d}$ be defined by
  \[
    \mu(r) \coloneqq
    \begin{pmatrix}
      \lambda_{r,1,1} & \dots & \lambda_{r,1,d} \\
      \vdots & \ddots & \vdots \\
      \lambda_{r,d,1} & \dots & \lambda_{r,d,d} \\
    \end{pmatrix}
    \qquad\text{and let}\qquad
    \mathbf a(n) \coloneqq (a_{1,n}, \dots,  a_{d,n}).
  \]
  Since $\Delta_r(f_i(z)) = \sum_{n=0}^\infty a_{i,kn+r}z^n$, we obtain 
  $\mathbf a(kn+r) =  \mathbf a(n) \mu(r)$ for $r \in \digits$. 
  Finally let $e_1 = (1,0,\ldots,0)^T  \in \Qbar^{d\times 1}$.
  Then $a_{1,[w]_k} = \mathbf a([w]_k)e_1 = \mathbf a(0) \mu(w) e_1$ for all words $w\in \digits^*$.
\end{proof}

\begin{lemma} \label{l:becker-linrep}
  Let $p_1(z),\ldots,p_d(z) \in \Qbar[z]$ with $e = \max\{ \deg p_1(z), \ldots, \deg p_d(z) \}$.
  If $f(z)=\sum_{n=0}^\infty a_n z^n \in \power{\Qbar}$ satisfies the $k$-Becker equation
  \[
    f(z) = p_1(z) f(z^k) + \cdots + p_d(z) f(z^{k^d})\,,
  \]
  then a linear representation for the $k$-regular sequence $(a_n)_{n\ge 0}$ is computable from 
  $a_0$ and $p_1(z),\ldots,p_d(z)$. 
\end{lemma}

\begin{proof}
  Following Becker \cite[Theorem 2]{becker94}, we see that the $\Qbar$-vector space $V$ 
  spanned by $\{\, z^i f(z^{k^j}) : 0\le i \le e,\, 0 \le j \le d \,\}$ is closed under all Cartier operators.
  Explicitly, if $r \in \digits$, $0\leq i \leq e$, and $j \ge 1$, then
  \[
    \Delta_r(z^i f(z^{k^j})) = \Delta_r(z^i) f(z^{k^{j-1}}) \in V,
  \]
  since $\Delta_r(z^i) = z^{(i-r)/k}$ if $i \equiv r \mod k$ and $\Delta_r(z^i)=0$ otherwise.
  If $j=0$, then $\deg(\Delta_r(z^i p_j(z))) \le 2e/k \le e$, and thus
  \[
    \Delta_r(z^i f(z)) = \Delta_r\Big( \sum_{j=1}^d z^i p_j(z) f(z^{k^j})\Big) = 
    \sum_{j=1}^d \Delta_r(z^i p_j(z)) f(z^{k^{j-1}}) \in V\, .
  \]
  Since $\Delta_r(z^i p_j(z))$ can be explicitly computed, we may apply \cref{l:relations-to-linrep} 
  to find a linear representation of $(a_n)_{n \ge 0}$.
  Since $0^i f(0^{k^j}) \in \{0, a_0\}$, the resulting linear representation only depends on 
  $p_1(z),\ldots,p_d(z)$ and $a_0$.
\end{proof}

It is rather non-trivial that the convolution product of $k$-regular sequences is again $k$-regular.
The standard way to show this uses the module-theoretic characterization of $k$-regularity; 
see \cite[Theorem 16.4.1]{allouche-shallit03} or \cite[Proposition 5.2.7]{berstel-reutenauer11}.
To see that a linear representation of the convolution product is computable from linear representations, 
we need to revisit this proof.

\begin{remark}
  Using the growth-based characterization of $k$-regular sequences in \cref{t-main:hn}, 
  it is easy to show that the convolution product of $k$-regular sequences is $k$-regular.
  However, since this characterization already makes use of this fact that convolution products of 
  $k$-regular sequences are $k$-regular (in \cref{p:hlog2}), this does not actually give a new, independent proof.
\end{remark}

\begin{lemma} \label{l:cauchy-linrep}
  Let $(a(n))_{n \ge 0}$ and $(b(n))_{n \ge 0}$ be two $k$-regular sequences in $\Qbar$, 
  each being given by a linear representation.
  Then a linear representation of the convolution product 
  $(a \star b(n))_{n \ge 0}$ is computable.
\end{lemma}

Let us recall that by definition $a \star b(n)=\sum_{i=0}^na(i)b(n-i)$. 

\begin{proof}
  For $r \in \digits$, let $(\Delta_r(a)(n))_{n \ge 0}$ be the sequence defined by 
  $\Delta_r(a)(n) \coloneqq a(kn+r)$. The key step in the proof of Allouche and Shallit 
  \cite[Theorem 16.4.1]{allouche-shallit03} is the reduction (with $r \in \digits$)
  \begin{equation} \label{eq:cauchy}
    \Delta_r(a\star b)(n) = \sum_{0 \le s \le r} \Delta_s(a) \star \Delta_{r-s}(b) (n) +  \sum_{r < s \le k-1}  
    \Delta_s(a) \star \Delta_{k+r-s}(b) (n-1)\, .
  \end{equation}

  We will adapt this proof to the case where $a$ and $b$ are given by linear representations.
  Without restriction we can assume that the linear representations of $a$ and $b$ both have 
  the same dimension $d \ge 0$. 
  After a change of basis, we may take the linear representation of $a$ to be $(\mathbf a(0),\kappa,e_1)$, 
  where $\mathbf a(0) =(a_1(0),\ldots,a_d(0))$ with $a_1(0)=a(0)$, 
  where $e_1 = (1, 0, \dots, 0)^T$, and where 
  $\kappa\colon \digits^* \to \Qbar^{d \times d}$ is a monoid homomorphism. 
  Define $\mathbf a(n) \coloneqq \mathbf a(0) \kappa(\langle n\rangle_k)$, where 
  $\langle n\rangle_k \in \digits^*$ is the canonical base-$k$ expansion of $n$.
  Then, in particular,
  \[
    (a_1(kn+r),\dots, a_d(kn+r)) =
    (a_1(n),\dots,a_d(n))\, \kappa(r) \qquad\text{for $r \in \digits$.}
  \]
  For $b$ we have a linear representation $(\mathbf b(0), \lambda, e_1)$ with analogous definitions.

  We construct a linear representation for $a \star b$ of dimension $2d^2$.
  For this, we index the first set of $d^2$ coordinates by $(i,j)$ in lexicographic order, 
  and the second by $(i',j')$, where $1 \le i,j \le d$. That is, the coordinates are indexed by 
  $(1,1)$, $(1,2)$, $\ldots\,$,~$(d,d)$, $(1',1')$, $(1',2')$, $\ldots\,$,~$(d',d')$.
  We use subscripts to indicate the corresponding entry of a matrix, e.g., $\kappa(r)_{i,j}$ is the entry in the $i$th row and $j$th column of $\kappa(r)$.
  For $1 \le i, j \le d$ and $r$,~$s \in \digits$, we get
  \[
    \Delta_r(a_i) \star \Delta_s(b_j) = \Big( \sum_{\ell=1}^d a_\ell\kappa(r)_{\ell,i} \Big) \star \Big( \sum_{m=1}^d b_m\lambda(s)_{m,j} \Big) = \sum_{\ell,m=1}^d  \kappa(r)_{\ell,i} \lambda(s)_{m,j} ( a_\ell  \star b_m ) \,.
  \]

  Using \cref{eq:cauchy},
  \[
    \begin{split}
      \Delta_r(a_i \star b_j)(n) &= \sum_{0 \le s \le r} \sum_{\ell,m=1}^d \kappa(s)_{\ell,i} \lambda(r-s)_{m,j} 
      \, ( a_\ell  \star b_m  )(n) \\
      &+  \sum_{r < s \le k-1}  \sum_{\ell,m=1}^d \kappa(s)_{\ell,i} \lambda(k+r-s)_{m,j}\,( a_\ell  \star b_m )(n-1)  \\
      &= \sum_{\ell,m=1}^d  \Big( \sum_{0 \le s \le r} \kappa(s)_{\ell,i} \lambda(r-s)_{m,j} \Big) ( a_\ell  \star b_m  )(n)\\
      &+  \sum_{\ell,m=1}^d \Big( \sum_{r < s \le k-1}  \kappa(s)_{\ell,i} \lambda(k+r-s)_{m,j} \Big)
       ( a_\ell  \star b_m )(n-1)\, .
    \end{split}
  \]
  Further note if $r \ge 1$, then $a_i \star b_j(kn+r-1) = \Delta_{r-1}(a_i \star b_j)(n)$.
  For $r=0$ we have $a_i \star b_j(kn-1) = a_i \star b_j(k(n-1) + (k-1)) = \Delta_{k-1}(a_i \star b_j)(n-1)$.
  In this case, in \cref{eq:cauchy}, the second sum vanishes, and we again obtain $a_i \star b_j(kn-1)$ as a linear combination of the $a_\ell\star b_m(n-1)$, namely,
  \[
    \Delta_{k-1}(a_i \star b_j)(n-1) = \sum_{\ell,m}^d \big(\sum_{0 \le s \le k-1} \kappa(s)_{\ell,i} \lambda(k-1-s)_{m,j} \big) 
    (a_\ell \star b_m)(n-1) .
  \]

  For two $d \times d$-matrices $A$,~$B$, the Kronecker product $A \otimes B$ is the 
  $d^2 \times d^2$-matrix defined by $(A\otimes B)_{(i,j),(\ell,m)} = A_{i,\ell}B_{j,m}$.
  For $r \in \digits$ with $r \ne 0$, we define the $2d^2 \times 2d^2$-matrix $\mu(r)$ by the block structure
  \[
    \mu(r) \coloneqq
    \begin{pmatrix}
      \sum_{0 \le s \le r} \kappa(s) \otimes \lambda(r-s)
         &  \sum_{0 \le s \le r} \kappa(s) \otimes \lambda(r-1-s)    \\
      \sum_{r < s \le k-1} \kappa(s) \otimes \lambda(k + r-s) 
         & \sum_{r < s \le k-1} \kappa(s) \otimes \lambda(k+r-1-s)
    \end{pmatrix}.
  \]
  Similarly
  \[
    \mu(0) \coloneqq
    \begin{pmatrix}
      \kappa(0) \otimes \lambda(0)
         &  0   \\
     \sum_{s=1}^{k-1} \kappa(s) \otimes \lambda(k-s)
         & \sum_{s=0}^{k-1} \kappa(s) \otimes \lambda(k-1-s)
    \end{pmatrix}.
  \]
  Define for each $n \ge 0$ the $2d^2$ row vector $\mathbf v(n)$ by $\mathbf v(n)_{(\ell,m)} = a_\ell \star b_m(n)$ 
  and $\mathbf v(n)_{(\ell',m')} = a_\ell \star b_m(n-1)$.
  Then
  \[
    \mathbf v(kn+r) = \mathbf v(n) \mu(r).
  \]
  Now $\mathbf v_{(1,1)}(n) = a_1 \star b_1(n) = a\star b(n)$.
  Thus, the triple $(\mathbf v(0), \mu, e_{(1,1)})$, where $e_{(1,1)}$ is the $2d^2$ column vector with $1$ in the coordinate $(1,1)$ and zeroes everywhere else, is a linear representation for $a \star b$.
\end{proof}

\begin{proposition}
  Let $f(z) = \sum_{n=0}^\infty a_n z^n \in \power{\Qbar}$ be $k$-regular, 
  given by a $k$-Mahler equation, the minimal $n_0 \ge 0$ with $a_{n_0} \ne 0$, 
  and the value $a_{n_0}$.
  Then a linear representation for the $k$-regular sequence $(a_n)_{n \ge 0}$ is computable.
\end{proposition}

\begin{proof}
  From a linear representation of $(a_n)_{n \ge n_0}$ it is easy to find one for $(a_n)_{n \ge 0}$.
  We may therefore without restriction assume $n_0=0$. 
  (An explicit $k$-Mahler equation for this power series can be found using \cite[Lemma 6.1]{adamczewski-bell17}.)

  Using Propositions \ref{p:decide-ru} and \ref{p:decide-rucommon}, we can further find a $k$-Mahler equation
  \[
    p_0(z) f(z) = p_1(z) f(z^k) + \cdots + p_d f(z^{k^d})\,,
  \]
  with $p_0(z),\ldots\,p_d(z) \in \Qbar[z]$, with $p_0(z)$ and $p_d(z)$ coprime, and with the property that 
  all roots of $p_0(z)$ are contained in $\rucommon$.
  In particular, we may assume $p_0(0)=1$. 

  Now \cref{t:dumas-structure} gives a decomposition
  \[
    f(z) = g(z) \Big( \prod_{i=0}^\infty p_0(z^{k^i}) \Big)^{-1}\,,
  \]
  where $g(z)$ is $k$-Becker, and a $k$-Becker equation for $g(z)$ can be computed.
  \Cref{l:becker-linrep} yields a linear representation for the coefficient sequence of $g(z)$. 
  Factoring $p_0(z)$ into linear factors of the form $1-z\zeta^{-1}$ with $\zeta \in \rucommon$, 
  we recall that $\prod_{i=0}^\infty (1 - z^{k^i} \zeta^{-1})^{-1}$ is $k$-regular.
  Indeed, by \cite[Proposition 7.8]{adamczewski-bell17}, this infinite product factors as a polynomial and a 
  $k$-Becker function (both computable). 
  Using \cref{l:cauchy-linrep} we find a linear representation for $\prod_{i=0}^\infty (1 - z^{k^i} \zeta^{-1})^{-1}$. 
  Finally, \cref{l:cauchy-linrep} allows us to find a linear representation for $f(z)$ itself.
\end{proof}

\subsection{Tame and finite semigroups}

From a linear representation, a minimal linear representation is computable, 
and we may now assume that the $k$-regular sequence $(a_n)_{n \ge 0}$ is given by 
such a minimal linear representation $(u,\mu,v)$. 
To decide which of Cases \ref{tms:hlog2}--\ref{tms:hloglog} of \cref{thm: hgt} 
the sequence belongs to, it now suffices to decide whether or not the finitely generated 
matrix semigroup $\mu(\digits^*)$ is finite, respectively, tame. 

For this, we first need the following two lemmas.

\begin{lemma} \label{l:inv-subspace}
  Let $A_1,\ldots,A_t \in \Qbar^{d \times d}$.
  It is possible to decide whether or not the matrices $A_1,\ldots,A_t$ have a proper 
  non-zero common invariant subspace, and if so, to compute one. 
\end{lemma}

\begin{proof}
  This can be done using exterior powers and Gr\"obner bases, see Arapura and Peterson \cite{arapura-peterson04}.
  A model-theoretic approach is given by Pastuszak in \cite{pastuszak17}.
  Both papers discuss the history of this problem.
\end{proof}

\begin{lemma} \label{l:bound-semigroup}
  Let $K$ be a number field and $d \ge 0$.
  For every $r \ge 0$, there exists a computable $n = n(r,K)$ with the following property: 
  if $\cS \subseteq K^{d \times d}$ is a \emph{finite} semigroup generated by $r$ matrices, 
  then $\card{\cS} \le n$.
\end{lemma}

\begin{proof}
  By a result of Mandel and Simon \cite[Theorem 1.2]{mandel-simon78} there exists such a bound $n(r,K,g)$, 
  that however also depends on the maximal size $g$ of a subgroup of $\cS$.
  Over a number field, Schur \cite{schur05} proved that there exists an explicit bound on the size of a finite subgroup of $\GL_d(K)$, so we can bound $g$ independently of $\cS$.
  (See also the, largely expository, article \cite{guralnick-lorenz06} for this and later results.) 
\end{proof}

\begin{proposition} \label{p:decide-semigroup}
  Let $\cS \subseteq \Qbar^{d \times d}$ be a finitely generated matrix semigroup, 
  given by a finite set of generators.
  It is decidable whether or not $\cS$ is 
  \begin{enumerate}
  \item \label{decide-semigroup:finite} finite,
  \item \label{decide-semigroup:tame} tame.
  \end{enumerate}
\end{proposition}

\begin{proof}
  Let $A_1$, $\ldots\,$,~$A_l$ be the given set of generators for $\cS$, and let $K$ be the number field generated by all the coefficients of the matrices $A_i$.
  Then $\cS \subseteq K^{d \times d}$ and it suffices to consider the problem over the field $K$.
  
  \ref*{decide-semigroup:finite}
  With the bound from \cref{l:bound-semigroup}, one can decide whether or not $\cS$ is finite.
  
  \ref*{decide-semigroup:tame}
  This problem can be reduced to the finiteness problem using \cref{l:structure-tame}.
  Indeed, by iterated application of \cref{l:inv-subspace}, we may decompose $K^{d \times 1} = V_1 \oplus \cdots \oplus V_s$ with each $V_i$ a $\cS$-invariant subspace that contains no proper, non-zero $\cS$-invariant subspace.
  Then \cref{l:structure-tame} implies that $\cS$ is tame if and only if $\cS|_{V_i}$ is finite for each $1 \le i \le s$.
  This can be decided using \ref*{decide-semigroup:finite}.
\end{proof}
\bibliographystyle{hyperalphaabbr}
\bibliography{height_mahler}

\end{document}